\documentclass[11pt, letterpaper]{amsart}

\renewcommand{\Re}{\operatorname{Re}}
\renewcommand{\Im}{\operatorname{Im}}

\newcommand{\defeq}{\stackrel{\rm{def}}{=}}

\usepackage{amssymb}
\usepackage{amsthm}
\usepackage{amsxtra}
\usepackage{graphicx}
\usepackage{color}
\usepackage{hyperref}
\newcommand{\R}{\mathbb R}
\newcommand{\C}{\mathbb C}

\newcommand{\eps}{\varepsilon}

\newcommand{\ds}{\displaystyle}

\newtheorem{theorem}{Theorem}[section]

\newtheorem{proposition}[theorem]{Proposition}
\newtheorem{lemma}[theorem]{Lemma}

\theoremstyle{remark}
\newtheorem{remark}[theorem]{Remark}

\numberwithin{equation}{section}
\begin{document}

\title[Scattering in NLS and gHartree equations]
{Scattering of radial data in the focusing  NLS\\
and generalized Hartree Equations}

\author[Anudeep K. Arora]{Anudeep Kumar Arora}
\address{Department of Mathematics \& Statistics \\
Florida International University, Miami, FL, USA}
\curraddr{}
\email{simplyandy7@gmail.com} %put your professional email
\thanks{}

%    The 2010 edition of the Mathematics Subject Classification is
%    now available.  https://mathscinet.ams.org/msc/msc2010.html
%    If you are citing a classification from the
%    new scheme, use the following input coding instead.
%\subjclass[2010]{Primary }


\date{}

\begin{abstract}
We consider the focusing nonlinear Schr\"odinger equation $i u_t + \Delta u  + |u|^{p-1}u=0$, $p>1,$ and the generalized Hartree equation
$iv_t + \Delta v  + (|x|^{-(N-\gamma)}\ast |v|^p)|v|^{p-2}u=0$, $p\geq2$, $\gamma<N$, in the mass-supercritical and energy-subcritical setting.
With the initial data $u_0\in H^1(\R^N)$ the characterization of solutions behavior under the mass-energy threshold is known for the NLS case from the works of Holmer and Roudenko in the radial \cite{HR08} and Duyckaerts, Holmer and Roudenko in the nonradial setting \cite{DHR08} and further generalizations (see \cite{AN13,  FXC11, G}); 
for the generalized Hartree case it is developed in \cite{AKAR}. In particular, scattering is proved following the road map developed by Kenig and Merle \cite{KM06}, using the concentration compactness and rigidity approach, which is now standard in the dispersive problems.

In this work we give an alternative proof of scattering for both NLS and gHartree equations in the radial setting in the inter-critical regime, following the approach of Dodson and Murphy \cite{DM17} for the focusing 3d cubic NLS equation, which relies on the scattering criterion of Tao \cite{Tao04}, combined with the radial Sobolev and Morawetz-type estimates. We first generalize it in the NLS case, and then extend it to the nonlocal Hartree-type potential. This method provides a simplified way to prove scattering, which may be useful in other contexts.
\end{abstract}

\maketitle

%    Text of article.

%\tableofcontents

\section{Introduction}

Consider two Cauchy problems: the focusing nonlinear Schr\"odinger (NLS) equation
\begin{align}\label{NLS}
\rm(NLS) \qquad
\begin{cases}
	&iu_t + \Delta u  + |u|^{p-1}u=0, \quad p > 1, ~ t \in \R, ~ x \in \R^N \\
	&u(x,0)=\,u_0(x)\in H^1(\R^N)
\end{cases}	
\end{align}
and the focusing Schr\"odinger-Hartree-type equation, which is also refereed to as the {\it generalized Hartree}, and abbreviated gHartree,
\begin{align}\label{gH}
\qquad \rm(gH) \quad
\begin{cases}
	&iv_t + \Delta v  + (|x|^{-(N-\gamma)}\ast |v|^p)|v|^{p-2} v=0, \quad t \in \R, ~ x \in \R^N \\
	&v(x,0)=\,v_0(x)\in H^1(\R^N)
\end{cases}
\end{align}
for $p\geq 2$ and $0<\gamma<N$. We consider $u=u(x,t)$ and $v=v(x,t)$ to be complex-valued functions in the equations \eqref{NLS} and \eqref{gH}, respectively.

Solutions to the equations \eqref{NLS} and \eqref{gH}, during their lifespan, conserve several quantities, including the mass, given (respectively, for NLS and gH) by
	\begin{align*}
		M_{NLS}[u(t)]\defeq\int_{\R^N}^{}|u(x,t)|^2\,dx=M_{NLS}[u_0],
	\end{align*}
	and
	\begin{align*}
	M_{gH}[v(t)]\defeq\int_{\R^N}^{}|v(x,t)|^2\,dx=M_{gH}[v_0].
	\end{align*}
The energy is also conserved, which is defined for \eqref{NLS} and \eqref{gH}, respectively, by
	\begin{align*}
		E_{NLS}[u(t)]\defeq\frac{1}{2}\int_{\R^N}^{}|\nabla u(x,t)|^2\,dx - \frac{1}{p+1}\int_{\R^N}^{}|u(x,t)|^{p+1}\,dx=E_{NLS}[u_0]
	\end{align*}
	and
	\begin{align*}
	E_{gH}[v(t)]\defeq\frac{1}{2}\int_{\R^N}^{}|\nabla v(x,t)|^2\,dx - \frac{1}{2p} \int_{\R^N}(|x|^{-(N-\gamma)}\ast|v(\,\cdot\,,t)|^p)|v(x,t)|^{p}\,dx=E_{gH}[v_0].
	\end{align*}

We omit the conservation of momentum, since we only consider radial solutions. The equations \eqref{NLS} and \eqref{gH} also enjoy several invariances, among them is scaling: if $u(x,t)$ solves \eqref{NLS}, then $\displaystyle{u^{\lambda}(x,t) = \lambda^{\frac{2}{p-1}}u(\lambda x,\lambda^2 t)} $ is also a solution to \eqref{NLS}, and similarly, if $v(x,t)$ solves \eqref{gH}, then $\displaystyle{v^{\lambda}(x,t) = \lambda^{\frac{\gamma+2}{2(p-1)}}v(\lambda x,\lambda^2 t)} $ is also a solution to \eqref{gH}. The $\dot{H}^s$ norm is invariant under the scaling for both equations. The scale-invariant Sobolev norm is $\dot{H}^s$ with
\begin{equation}\label{E:scaling}
\displaystyle{s=\frac{N}{2}-\frac{2}{p-1}}\quad \mbox{for} ~ \eqref{NLS} \quad \mbox{and} \qquad {s=\frac{N}{2}-\frac{\gamma+2}{2(p-1)}} \quad \mbox{for} ~\eqref{gH}.
\end{equation}
In this paper we will consider the equations \eqref{NLS} and \eqref{gH} with the nonlinearity power $p$ such that the equations are energy-subcritical, $s < 1$.

Using Duhamel's formula, we can write \eqref{NLS} in the integral form
	\begin{align}\label{duhamel_NLS}
	u(t) = e^{it\Delta}u_0 + i \int_0^t e^{i(t-s)\Delta} |u|^{p-1}u(s)\,ds,
	\end{align}
	and the corresponding Duhamel formulation for \eqref{gH} is given by
	\begin{align}\label{duhamel_gH}
	v(t) = e^{it\Delta}v_0 + i \int_0^t e^{i(t-s)\Delta} (|x|^{-(N-\gamma)}\ast |v|^p)|v|^{p-2}v(s)\,ds.
	\end{align}

The local well-posedness of solutions to \eqref{NLS} or \eqref{gH} is obtained via a fixed point theorem, or a contraction on the map defined from the Duhamel's formulas \eqref{duhamel_NLS} and \eqref{duhamel_gH}, correspondingly. For the purpose of this paper, we only need the local well-posedness in $H^1$, which is long known for the standard NLS and the standard $(p=2)$ Hartree equations from the works of Cazenave \cite{C03} and Ginibre \& Velo \cite{GV80}. In the general case of the gHartree $(p> 2)$ equation, the local well-posedness is given in our work \cite{AKAR}.

Denote the maximal existence (in time) interval of solutions to \eqref{NLS} and \eqref{gH} by $(T_*, T^*)$. We say a solution is global in forward time if $T^* = +\infty$; similarly, if $T_* = -\infty$, the solution is global in backward time. If both $T_*$ and $T^*$ are infinite, then the solution is global.
	
We say that a solution $u(t)$ to \eqref{NLS} scatters in $H^s(\R^N)$, $s\geq 0$, as $t\rightarrow +\infty$, or correspondingly, $v(t)$ to \eqref{gH} scatters in $H^s$, if there exists  $u^+ \in H^s(\R^N)$, or $v^+\in H^s(\R^N)$, such that
$$
\lim\limits_{t\rightarrow +\infty}\|u(t)-e^{it\Delta}u^+\|_{H^s(\R^N)}=0.
$$
In this paper we investigate scattering in $H^1$.

We consider both equations in the inter-critical regime such that $0< s < 1$ with $s$ defined in \eqref{E:scaling}, and provided $p\geq 2$ for the gHartree equation. In this case, both  equations \eqref{NLS} and \eqref{gH} admit solutions of the form $\displaystyle e^{it}Q(x)$, which are global but non-scattering, where $Q$ solves in the NLS case the following nonlinear elliptic equation %in the NLS case
\begin{align}\label{NLSQeq}
-Q + \Delta Q + |Q|^{p-1} Q = 0,
\end{align}
and in the gHartree case $Q$ solves the Choquard equation
\begin{align}\label{gHQeq}
-Q + \Delta Q + \left(|x|^{-(N-\gamma)}\ast |Q|^p \right)|Q|^{p-2} Q = 0.
\end{align}
The equation \eqref{NLSQeq} has countably many $H^1$ (real) solutions. Among those, there is exactly one solution of minimal mass, called the ground state, which is positive, radial, and exponentially decaying (e.g., see Berestycki \& Lions \cite{BL83I, BL83II},  Kwong \cite{K89}; for a review, for example, see Tao \cite[Appendix B]{Tao06}).
	
For the equation \eqref{gHQeq} the existence and uniqueness of the positive solution in the standard 3d Choquard equation with $p=2$ and $\gamma=2$ was first proved by Lieb \cite{Lieb77} (see also discussions on existence of solutions by P. L. Lions \cite{Lions80} and \cite{Lions82}). This solution is smooth, radial, monotonically decreasing (in the radial coordinate) and has exponential decay to infinity; similar to the NLS case, it also minimizes the mass among all $H^1$ real solutions of \eqref{gHQeq}. The existence and uniqueness for the 4d Choquard equation (also with $p=2$ and $\gamma = 2$) was shown by Krieger, Lenzmann \& Rapha\"el \cite{KLR09}, generalizing Lieb's approach for the uniqueness (the existence follows, for example, from the standard variational arguments). The generalization of the uniqueness proof for $p=2$, $\gamma=2$ and other dimensions is given in the appendix section of our work \cite{AKAR}. When $\gamma=2$, $N=3$ and $p=2+\varepsilon$, the uniqueness is proved recently by Xiang in \cite{X16}.
The existence of positive solutions along with the regularity, and radial symmetry of solutions to \eqref{gHQeq} for $\frac{N+\gamma}{N}<p<\frac{N+\gamma}{N-2}$ with $0<\gamma<N$ was shown by Moroz \& Schaftingen \cite{MS13} (see also a good review by Moroz \& Schaftingen \cite{MS17}). The uniqueness in a general case ($\gamma \neq 2$ and $2 \leq p < (N+\gamma)/(N-2)$ or $s<1$) is a more delicate issue, and, in general, is not known. We note that the proof of uniqueness of a positive solution in Hartree case, Lieb \cite{Lieb77}, is quite different from the corresponding result for the NLS (for example, as given by Kwong \cite{K89}), and relies on the Newton's potential representation of the nonlinearity. In the cases when uniqueness is known, we denote this unique positive solution, or the ground state, by $Q$ for both \eqref{NLSQeq} and \eqref{gHQeq}. In the gHartree cases when the uniqueness is not available, it is sufficient to use the minimizer of Gagliardo-Nirenberg inequality and its value expressed via $\|Q\|_{L^2}$, see \cite{AKAR}.

To characterize the global behavior of solutions in the NLS equation, Holmer and Roudenko in \cite{HR07} observed that the quantities $M[u]^{1-s}E[u]^s$ and $\|u\|_{L^2}^{1-s}\|\nabla u\|_{L^2}^s$ are scale-invariant and scale as the $\dot{H}^s$ norm. In \cite{HR08} they proved the following result, using the concentration - compactness and rigidity road map of Kenig and Merle \cite{KM06}, in the case of the focusing $3d$ cubic NLS equation (in the radial setting)
	\begin{align}\label{3dcubic}
		\begin{cases}
		&iu_t+\Delta u +|u|^2u=0;\quad t\in \R,\,\,\,x\in\R^3\\
		&u(x,0)=\,u_0(x)\in H^1(\R^3).
		\end{cases}
	\end{align}
	\begin{theorem}\label{HR}
		Let $u_0\in H^1(\mathbb R^3)$ be radial and $u(t)$ be the corresponding solution to \eqref{3dcubic} in $H^1(\mathbb{R}^3)$. Suppose $M[u_0]E[u_0]<M[Q]E[Q]$. If $\| u_0 \|_{L^2} \| \nabla u_0 \|_{L^2} < \|Q\|_{L^2} \| \nabla Q \|_{L^2}$, then the solution to \eqref{3dcubic} is global and scatters in $H^1$.
\end{theorem}
\begin{remark}
Their result also contains the finite time blow-up conclusion in the case when $\|u_0\|_{L^2} \|\nabla u_0\|_{L^2}> \|Q\|_{L^2} \|\nabla Q \|_{L^2}$, however, we omit that part as it is not needed for this paper.
\end{remark}
\begin{remark}
Later Duyckaerts, Holmer \& Roudenko in \cite{DHR08} extended their result to the general, non-radial setting; and since we consider only the radial case, we omit the general case as well. %with this approach we only can handle radially symmetric solutions.
\end{remark}

Recently, Dodson \& Murphy in \cite{DM17} presented a simplified proof of Theorem \ref{HR} that avoids concentration-compactness route. They used a scattering criterion introduced by Tao in \cite{Tao04}, which together with the radial Sobolev embedding and virial/Morawetz estimate was sufficient to prove (in the radial setting) Theorem \ref{HR}. Even more recently, Dodson \& Murphy extended their method to the non-radial case in \cite{DM18}, also avoiding the concentration-compactness.
	
The purpose of this work is to generalize the method of Dodson \& Murphy \cite{DM17} in the radial case to the inter-critical range of the nonlinear Schr\"odinger equation \eqref{NLS} and also show that it can be applied in the case of the nonlocal potential such as in the generalized Hartree equation \eqref{gH}. Our result is a new (or an alternative) proof of the following two theorems.

\begin{theorem}[Scattering in NLS]\label{main1}
Consider the NLS equation \eqref{NLS} with 
$N > 2$, and $1+\frac{4}{N}< p<1+\frac{4}{N-2}$ $(0< s<1)$. Let $u_0\in H^1 (\mathbb{R}^N)$ be radial and 
assume 
$$
M[u_0]^{1-s}E[u_0]^s < M[Q]^{1-s} E[Q]^s.
$$
If		
$$
\|u_0\|^{1-s}_{L^2(\mathbb{R}^N)}\|\nabla u_0\|^s_{L^2(\mathbb{R}^N)}< \|Q\|^{1-s}_{L^2(\mathbb{R}^N)} \|\nabla Q \|^s_{L^2(\mathbb{R}^N)},
$$
then the solution $u(t)$ to \eqref{NLS} is global and scatters in $H^1(\mathbb{R}^N)$.
\end{theorem}
	
\begin{theorem}[Scattering in gH]\label{main2}
Consider the gHartree equation \eqref{gH} with
$N > 2$, $p\geq 2$ and $ 1+ \frac{\gamma+2}{N}< p<1+\frac{\gamma+2}{N-2}$ $(0< s<1)$. 
Let $v_0\in H^1 (\mathbb{R}^N)$ be radial and
assume
$$
M[v_0]^{1-s}E[v_0]^s < M[Q]^{1-s} E[Q]^s.
$$
If
$$
\|v_0 \|^{1-s}_{L^2(\mathbb{R}^N)} \| \nabla v_0 \|^s_{L^2(\mathbb{R}^N)} < \| Q \|^{1-s}_{L^2(\mathbb{R}^N)} \| \nabla Q \|^s_{L^2(\mathbb{R}^N)},
$$
then the solution $v(t)$ to \eqref{gH}  is global and scatters in $H^1(\mathbb{R}^N)$.
\end{theorem}
\begin{remark}\label{N3}
We only consider $N > 2$ as we use the dispersive estimate \eqref{dispest} in \eqref{Ng2}, which gives the logarithmic divergence of the integral when $N=2$. 
\end{remark}
\begin{remark}
We do not cover the case $s=1$ (energy-critical) as this approach takes into account an {\it a priori} uniform bound on $\dot{H}^1$ norm of a solution in terms of the energy, and having the gap between the critical index $s_c$ as defined in \eqref{E:scaling}  and $s=1$ is an essential part of the proof. (However, it would be possible to cover this case given an \textit{a priori} uniform $\dot{H}^{s_1}$ bound and consider $s_c<s_1$.)
%, thus, to treat the $\dot{H}^1$ critical case with this method, one would require an additional smoothness.
\end{remark}
To prove the theorems we establish Morawetz estimates for both equations in the inter-critical regime by employing the radial Sobolev inequality. This implies that the potential energy escapes as $t\rightarrow\infty$, which in turn yields spreading of the mass. To obtain the scattering and conclude the proofs of Theorems \ref{main1} and \ref{main2}, we generalize the scattering criterion of Tao from \cite{Tao04} (for the 3d cubic NLS) to all inter-critical cases of NLS and also obtain the scattering criterion for the gHartree equation.

\begin{remark}
Theorem \ref{main2} was also proved via the concentration-compactness method in \cite{AKAR}. Nevertheless, we show that in the radial case this new approach can also be applied in the case of the nonlocal convolution nonlinearity, i.e., for the gHartree equation. 
\end{remark}
We also note that the treatment of the \textit{gHartree} case is different from the NLS case, in particular, 
\begin{itemize}
\item%[(i)] 
the estimate of inhomogeneous term in the Duhamel formula via Strichartz estimate in Lemma \ref{Scat_crit_gH} (using Lemma \ref{HLS} to handle the convolution term), and
\item%[(ii)] 
most importantly, in the Morawetz estimate for gHartree, one expects (as in the NLS case) to obtain the upper bound on the potential term, $P(v)$ which is of convolution type given by $P(v) =\int_{\R^N}(|x|^{-(N-\gamma)}\ast|v(\,\cdot\,,t)|^p)|v(x,t)|^{p}\,dx$. It is not straightforward to rule out the concentration of mass from the potential (convolution) term in the gHartree case. Thus, when deriving the Morawetz estimate (Proposition \ref{morawetz_est_gH}), we rely on the  $L^{\frac{2Np}{N+\gamma}}_x$ norm. This is possible since in Lemma \ref{coercivityballgH_new}, we bound the virial from below using Sobolev inequality (thanks to the assumption $s<1$, i.e., $p<\frac{N+\gamma}{N-2}$) by $L^{\frac{2Np}{N+\gamma}}_x$ norm
$$
\|\nabla v\|_{L^2(\R^N)}^2-\frac{s(p-1)+1}{p}P(v)\geq \delta \|\nabla v\|^2_{L^2(\R^N)}\geq \tilde{\delta}\|v\|_{L^{\frac{2Np}{N+\gamma}}_x(\R^N)}
$$
 instead of using the $P(v)$, convolution type potential term. Observing that $\frac{2Np}{N+\gamma}>2$ allows us to obtain the mass evacuation to conclude scattering.
\end{itemize}

The paper is organized as follows: in Section \ref{pre}, we review Strichartz estimates and relevant background. We extend Tao's scattering criteria to a general case of \eqref{NLS} and \eqref{gH} in Lemma \ref{Scat_crit} and Lemma \ref{Scat_crit_gH}, respectively. In Section \ref{var}, we review some properties of the corresponding ground states $Q$ and prove coercivity estimates for the equation \eqref{NLS} in Section \ref{variational} and for the equation \eqref{gH} in Section \ref{variational_gH}. Finally, in Section \ref{mainNLS} and Section \ref{maingH}, we prove Theorem \ref{main1} and Theorem \ref{main2}, respectively. For that we use the Morawetz estimate (Proposition \ref{morawetz_est} for \eqref{NLS} and Proposition \ref{morawetz_est_gH} for \eqref{gH}) together with the respective energy evacuation (Proposition \ref{energyevac} and Proposition \ref{energyevac_gH}) and scattering criteria (Lemma \ref{Scat_crit} and Lemma \ref{Scat_crit_gH}).
	
At the time of publication, an article following the method of Dodson-Murphy \cite{DM17} for the inhomogeneous NLS (INLS) was done, from which scattering can be deduce for the standard NLS, see \cite{LC}.
	
\textbf{Acknowledgments.} This work is a part of the author's PhD thesis under the guidance of Svetlana Roudenko. A.K.A would like to thank her %Svetlana Roudenko 
for suggesting the topic and countless discussions and help on this paper. The author would also like to thank Jason Murphy and Ben Dodson for their valuable insights and conversations on the subject.
Research of this project was partially funded by S. Roudenko's NSF CAREER grant DMS-1151618, from which A.K.A. was partially supported by the graduate research fellowship.

\section{Preliminaries}\label{pre}
Recall the linear Schr\"odinger evolution from initial data $f_0$
\begin{align}\label{linsol}
f(x,t)=e^{it\Delta}\,f_0(x) =\frac{C}{t^{N/2}}\int_{\R^N}^{}e^{i\frac{|x-y|^2}{4t}}f_0(y)\,dy,
\end{align}
and also the dispersive estimate, $2\leq p\leq \infty$ and $\frac{1}{p}+\frac{1}{p'}=1$,
	\begin{align}\label{dispest}
	\|e^{it\Delta}f_0\|_{L^p(\R^N)}\lesssim|t|^{-\left(\frac{N}{2}-\frac{N}{p}\right)}\|f_0\|_{L^{p'}(\R^N)}.
	\end{align}
	We have the Strichartz inequality
	\begin{align}\label{stri1}
	\|e^{it\Delta}\,f_0\|_{L_{t,x}^{\frac{2(N+2)}{N}}(\R\times\R^N)}\lesssim \|f_0\|_{L^2(\R^N)}
	\end{align}
	and since $e^{it\Delta}$ commutes with derivatives, we have
	\begin{align}\label{stri2}
	\|e^{it\Delta}\,f_0\|_{L_t^{\frac{2(N+2)}{N}}W_x^{s,\frac{2(N+2)}{N}}(\R\times\R^N)}\lesssim \|f_0\|_{H^s(\R^N)}.
	\end{align}
	From Sobolev embedding for $0\leq s\leq 1$, we observe that
	\begin{align}\label{stri3}
		\|e^{it\Delta}\,f_0\|_{L_t^{\frac{2(N+2)}{N}}L_x^{\frac{2N(N+2)}{N^2-2s(N+2)}}(\R\times\R^N)}\lesssim \|f_0\|_{H^s(\R^N)}.
	\end{align}
		We also recall the inhomogeneous version of the above estimate 
		%in which one introduces a forcing term in a dual Strichartz space
			\begin{align}\label{stri4}
	\|u\|_{L_{t,x}^{\frac{2(N+2)}{N}}([t_0,\infty)\times\R^N)}\lesssim\|u(t_0)\|_{L^2(\R^N)}
			+\|F(u)\|_{L_{t,x}^{\frac{2(N+2)}{N+4}}([t_0,\infty)\times\R^N)},
		\end{align}
		where $u$ is a solution to $iu_t + \Delta u = F(u)$. Next, we recall the Hardy-Littlewood-Sobolev (HLS) estimate, which will be used to handle the nonlocal (convolution) term in the generalized Hartree \eqref{gH} setting.

\begin{lemma}[HLS inequality, \cite{Lieb83}]\label{HLS}
For $0<\gamma <N$ and $p>1$, there exists a sharp constant $C=C(N,p,\gamma)>0$ such that
$$
\left\|\int_{\R^N}^{}\frac{u(y)}{|x-y|^{N-\gamma}}\,dy\right\|_{L^q(\R^N)}\leq C\|u\|_{L^p(\R^N)},
$$
where $\frac{1}{q}=\frac{1}{p}-\frac{\gamma}{N}$ and $p<\frac{N}{\gamma}$. 
\end{lemma}
	
\begin{remark}\label{reisz}
Observe that 
$$
\nabla\left(|x|^{-(N-\gamma)}\right)=C_{N,\gamma}|x|^{-(N-(\gamma-1))}.
$$
\end{remark}

We give the local-in-time analogue of the Strichartz estimates.
\begin{lemma}\label{smallness}
Let $u$ be a solution to \eqref{NLS} or \eqref{gH} satisfying
	\begin{align}\label{unibnd}
		\sup_{t\in[0,\infty)}\|u(t)\|_{H^1(\R^N)}\leq E.
	\end{align}
Then for $2\leq p\leq \frac{2N}{N-2}$
\begin{equation}\label{smallbnd1}
		\sup_{t\in[0,\infty)}\|u(t)\|_{L^p(\R^N)}\lesssim 1,
\end{equation}
and for any $T\geq 0$ and $\tau >0$
\begin{align}\label{smallbnd2}
\|\nabla u\|_{L_t^{\frac{2(N+2)}{N}}L_x^{\frac{2(N+2)}{N}}([T,T+\tau]\times\R^N)}\lesssim\langle\tau\rangle^{\frac{N}{2(N+2)}}.
\end{align}
Furthermore, %we have
\begin{align}\label{smallbnd3}
\|u\|_{L_t^{\frac{2(N+2)}{N}}L_x^{\frac{2N(N+2)}{N^2-2s(N+2)}}([T,T+\tau]\times\R^N)}\lesssim\langle\tau\rangle^{\frac{N}{2(N+2)}}
\end{align}
and
\begin{align}\label{smallbnd4}
\|u\|_{L_{t,x}^{\frac{2(N+2)}{N-2}}([T,T+\tau]\times\R^N)}\lesssim\langle\tau\rangle^{\frac{N-2}{2(N+2)}}
\end{align}
for $T\geq 0$ and $\tau >0$. 
\end{lemma}

\begin{proof} The proof follows a standard argument, which can be found in \cite{C03}. For completeness we outline the proof and provide the estimates for gHartree at the end.  First note that in the NLS case, the inequality \eqref{smallbnd1} is a straightforward application of the Sobolev embedding \eqref{unibnd}. In order to get 
\eqref{smallbnd2}, take $T>0$ and fix $\tau>0$. Since both $\|u(t)\|_{H^1}$ and $\| u(t)\|_{L^p}$, $2\leq p \leq \frac{2N}{N-2}$, are bounded uniformly in time by \eqref{unibnd} and \eqref{smallbnd1}, the value of $u(T)$ can be taken as the initial data with the solution being well-defined on the interval $I_\tau=[T, T+\tau]$.  
We show \eqref{smallbnd2} and subsequent inequalities  
%\eqref{smallbnd3} and \eqref{smallbnd4} 
for $\tau >0$ sufficiently small. % depending on $E$. 
Then the claim for larger times follows by %standard continuity argument; 
decomposing the time interval into smaller intervals.  

For any $T>0$ and $\tau>0$, define the space\footnote{Note that the pair $\left(\frac{2(N+2)}{N-2},\frac{2(N+2)}{N-2}\right)$ is $\dot{H}^1$-admissible, and $\left(\frac{2(N+2)}{N}, \frac{2N(N+2)}{N^2-2s(N+2)}\right)$ is $\dot{H}^s$-admissible.} 
\begin{align*}
X_\tau = \bigg\{u\in L_t^{\infty}H_x^1 &{(I_\tau\times\R^N)}: \,(1+\partial_x)\,u\in L_{t,x}^{\frac{2(N+2)}{N}}{(I_\tau\times\R^N)},\\
&u\in L_t^{\frac{2(N+2)}{N}}L_x^{\frac{2N(N+2)}{N^2-2s(N+2)}}{(I_\tau\times\R^N)}\,\cap \, L_{t,x}^{\frac{2(N+2)}{N-2}}{(I_\tau\times\R^N)} \bigg\}.
\end{align*}
Assume that $\tau$ is sufficiently small so that $\ds \|u\|_{X_\tau} < \eps<1$, we obtain		
\begin{align*}\label{gHref}
\|u\|_{X_\tau}&\lesssim \|u(T)\|_{H^1(\R^N)}+\|u\|^{(1-s)(p-1)}_{L_{I_\tau,x}^{\frac{2(N+2)}{N}}}\|u\|^{s(p-1)}_{L_{I_\tau,x}^{\frac{2(N+2)}{N-2}}}\|\nabla u\|_{L_{I_\tau,x}^{\frac{2(N+2)}{N}}}\lesssim \|u(T)\|_{H^1(\R^N)}+\|u\|^p_{X_\tau}.
\end{align*}
Using the smallness of $\|u\|_{X_\tau}$, we deduce $\|u\|_{X_\tau} \leq c_{N,p} \|u(T)\|_{H^1(\R^N)}$. In particular, 
\begin{equation*}
\|\nabla u\|_{L_{I_\tau,x}^{\frac{2(N+2)}{N}}}^{\frac{2(N+2)}{N}}
=\int_{T}^{T+\tau}\|\nabla u\|^{\frac{2(N+2)}{N}}_{L_x^{\frac{2(N+2)}{N}}}\,d\tau
\leq c \, \tau,
\end{equation*} 
and taking $\frac{2(N+2)}{N}$th power of both sides, we get \eqref{smallbnd2}.
 
Now for any $\tau$, we subdivide the interval $[T, T+\tau]$ into the intervals $I_j$ of length $\tau_0$ (that is $I_j = [T+j\tau_0, T+(j+1)\tau_0]$, $0 \leq j \leq k$ ) such that $\ds \|u\|_{X_{I_j}} \lesssim \eps $, where $\eps = \eps(E)$, this can be done because of the assumption \eqref{unibnd} (for a similar argument see \cite{CR07}).
Let $k=\left[\frac{\tau}{\tau_0} \right]$. We run the above argument on each $I_j$ 
%= [T+j\tau_0, T+(j+1)\tau_0]$, $0 \leq j \leq k$, 
obtaining $\ds \|u\|_{X_{I_j}} \leq c_{N,p} \|u(T+j\tau_0)\|_{H^1(\R^N)}$. Finally, adding all the pieces together and recalling the bound \eqref{unibnd}, we get
\begin{align*}
	\|\nabla u\|_{L_{I_\tau,x}^{\frac{2(N+2)}{N}}}^{\frac{2(N+2)}{N}}\lesssim\sum_{j=0}^{k+1}\int_{T+j\tau_0}^{T+(j+1)\tau_0}\|\nabla u\|^{\frac{2(N+2)}{N}}_{L_x^{\frac{2(N+2)}{N}}}\,d\tau\leq c(N,p,E) \epsilon^{\frac{2(N+2)}{N}} (1+k)\,\,\lesssim_E\,\,\langle\tau\rangle,
\end{align*}
which implies the desired estimate. In a similar fashion we obtain the bounds
\begin{align*}
\|u\|_{L_t^{\frac{2(N+2)}{N}}L_x^{\frac{2N(N+2)}{N^2-2s(N+2)}}}\,\,\lesssim_{E}\,\,\langle\tau\rangle^{\frac{N}{2(N+2)}}\quad
\text{and}\quad
\|u\|_{L_{t,x}^{\frac{2(N+2)}{N-2}}}\,\,\lesssim_{E}\,\,\langle\tau\rangle^{\frac{N-2}{2(N+2)}}.
\end{align*}

The argument for the gHartree equation \eqref{gH} follows a similar strategy as for the NLS except the use of a different $L^2$ dual pair, since there is an additional step involved because of the convolution in the nonlinear term. Thus,  
\begin{align}\label{gHref1}\notag
	\|v\|_{X_\tau}\lesssim\|\nabla v\|_{L_{I_\tau,x}^{\frac{2(N+2)}{N}}}\lesssim&\|v(T)\|_{H^1(\R^N)}\\
	&+\|\nabla\big((|x|^{-(N-\gamma)}\ast|v|^p)|v|^{p-2}v\big)\|_{L_{I_\tau}^\frac{2(N+2)}{N+2(\gamma+2)}L_x^{\frac{2N(N+2)}{N^2+4(N-\gamma)}}}
\end{align}
To estimate the second term in \eqref{gHref1}, we first use H\"older's and product rule to obtain
\begin{align*}
\|\nabla\big(&(|x|^{-(N-\gamma)}\ast|v|^p)|u|^{p-2}v\big)\|_{L_{I_\tau}^\frac{2(N+2)}{N+2(\gamma+2)}L_x^{\frac{2N(N+2)}{N^2+4(N-\gamma)}}}\\
	\lesssim&\,\|\nabla\left(|x|^{-(N-\gamma)}\ast|v|^p\right)\|_{L_{I_\tau}^{\frac{2(N+2)}{N+\gamma+2}}L^{\frac{2N(N+2)}{N(N+\gamma+2)-2\gamma(N+2)}}_{x}}\||v|^{p-2}v\|_{L^{\frac{2(N+2)}{\gamma+2}}_{I_\tau,x}}\\
	\,+&\|\left(|x|^{-(N-\gamma)}\ast|v|^p\right)\|_{L_{I_\tau}^{\frac{2(N+2)}{(N-2s)p}}L^{\frac{2N(N+2)}{N(N-2s)p-2\gamma(N+2)}}_{x}}\|\nabla\left(|v|^{p-2}v\right)\|_{L^{\frac{2(N+2)}{N+(N-2s)(p-2)}}_{I_\tau,x}},
\end{align*}
where we used \eqref{E:scaling} to deduce that $(N-2s)(p-1)=\gamma+2$. Now, applying Lemma \ref{HLS}, we continue bounding the above expression
\begin{align*}
	\lesssim&\,\|\nabla\left(|v|^p\right)\|_{L^{\frac{2(N+2)}{N+\gamma+2}}_{I_\tau,x}}\|v\|^{p-1}_{L^{\frac{2(N+2)}{N-2s}}_{I_\tau,x}}+\|v\|^p_{L^{\frac{2(N+2)}{N-2s}}_{t,x}}\|v\|^{p-2}_{L^{\frac{2(N+2)}{N-2s}}_{I_\tau,x}}\|\nabla v\|_{L^{\frac{2(N+2)}{N}}_{I_\tau,x}}\\
	\lesssim&\,\|\nabla v\|_{L^{\frac{2(N+2)}{N}}_{I_\tau,x}}\|v\|^{p-1}_{L^{\frac{2(N+2)}{N-2s}}_{I_\tau,x}}\|v\|^{p-1}_{L^{\frac{2(N+2)}{N-2s}}_{I_\tau,x}}+\|v\|^{2(p-1)}_{L^{\frac{2(N+2)}{N-2s}}_x}\|\nabla v\|_{L^{\frac{2(N+2)}{N}}_{I_\tau,x}}\\
	\lesssim&\,\|v\|^{2(p-1)}_{L^{\frac{2(N+2)}{N-2s}}_{I_\tau,x}}\|\nabla v\|_{L^{\frac{2(N+2)}{N}}_{I_\tau,x}},
\end{align*}
where for the second inequality we used $\nabla(|v|^p)=c|v|^{p-1}\nabla v$ and then H\"older's for the norm $L^{\frac{2(N+2)}{N+\gamma+2}}_{I_\tau,x}$, and the third inequality follows from the Sobolev embedding. The remaining proof follows the same steps.
\end{proof}

Next we recall the radial Sobolev inequality, which is one of the key ingredients in %the work of Dodson-Murphy 
\cite{DM17}.

\begin{lemma}[radial Sobolev, % embedding, 
\cite{S77}]\label{radsob} 
For $f\in H^1_{rad}(\R^N)$, $N\geq 2$ and $R>0$, we have
$$
\sup_{|x|\geq R} |f(x)|\leq \frac{1}{R^{\frac{N-1}{2}}} \|f\|_{L^2(\R^N)}^{\frac{1}{2}} \|\nabla f \|_{L^2(\R^N)}^{\frac{1}{2}}.
$$
\end{lemma}	
We now obtain a scattering criterion for the NLS for all $s\in(0,1)$, which generalizes Tao's scattering criterion for the 3d cubic NLS for radial solutions ( Theorem 1.1 in \cite{Tao04}). Since the use of Strichartz pairs is delicate, we provide a sketch for NLS and afterwards obtain the estimates for the gHartree equation \eqref{gH} in Lemma \ref{Scat_crit_gH}.

\begin{lemma}[Scattering criterion for NLS]\label{Scat_crit}
Consider $0<s<1$. Suppose $u$ is a radial solution to \eqref{NLS} satisfying
$$
\sup_{t\in[0,\infty)} \|u(t) \|_{H^1(\R^N)}\leq E\quad\text{with}\quad E>0.
$$
If there exist $\varepsilon>0$ and $R >0$ depending only on %the energy 
$E$ such that
\begin{equation}\label{masscond}
\liminf_{t\rightarrow\infty}\int_{|x|\leq R}^{}|u(x,t)|^2\,dx\leq\varepsilon,
\end{equation}
then $u(t)$ scatters forward in time in $H^1$.
\end{lemma}
\begin{remark}
	The proof below uses a symmetric pair in \eqref{OM-change} as in the scattering criterion in the nonradial paper by Dodson-Murphy \cite{DM18}, except for a different dual pair in \eqref{OM-change3}, restricting the range to $0<s<1/2$. For the entire intercritical range $0<s<1$, the proof is in Appendix \ref{OM}, accommodating the inhomogeneous Strichartz estimate from Foschi \cite[Theorem 1.4]{F05}.      
\end{remark}
\begin{proof}
Let $0<\eps <1$ be a small constant and $R(\varepsilon)\gg 1$ be a large number, both to be chosen later. From \eqref{stri3}, we have
\begin{equation}\label{OM-change}
	\|e^{it\Delta}\,u_0 \|_{L^{\frac{2(N+2)}{N-2s}}_{t,x}(\R\times\R^N)}\lesssim 1.
\end{equation}
By monotone convergence we may find a (large enough) time $T_0 > \varepsilon^{-\frac{4(1-s)(N(p-1)-2)}{N}}> 1$ depending on $u$ such that
\begin{align}\label{linearcomp}
\|e^{it\Delta}\,u_0\|_{L^{\frac{2(N+2)}{N-2s}}_{t,x}([T_0,\infty)\times\R^N)} \lesssim \eps.	
\end{align}
By the hypothesis, we may choose $T_1>T_0$ so that
$$
\int_{|x|\leq R}^{}|u(x,T_1)|^2\,dx\lesssim \varepsilon. 
$$
We denote by $\chi$ a smooth, radially symmetric function on $\R^N$ with supp $\chi\subset B(0, 1)$, which equals $1$ on $B(0, 1/2)$. For any $R>0$, we define $\chi_R(x) = \chi(x/R)$, noting that $\chi_R =1$ on $B(0,R/2)$. Then, we deduce from \eqref{unibnd}
$$
\left|\partial_t\int\chi_R(x)|u(x,t)|^2\,dx\right|\lesssim\frac{1}{R},
$$
since $ \|\nabla\chi_R \|_{L^{\infty}(\R)} = O(1/R)$.
Set $0<\alpha < \frac{2(1-s)(N(p-1)-2)}{N}$ 
(observe that $\frac{2(1-s)(N(p-1)-2)}{N}>0$, since $s<1$ and $p>1+\frac{2}{N}$ which is guaranteed from $s>0$) and note that $T_1-\eps^{-\alpha}>0$. Then for large enough $R=R(\varepsilon)\gg 0$ we observe that
\begin{equation}\label{localmass}
\sup_{t\in[T_1-\varepsilon^{-\alpha},T_1]}\int\chi_R(x)|u(x,t)|^2\,dx=\|\chi_Ru\|_{L^{\infty}_tL^2_x([T_1-\varepsilon^{-\alpha},T_1]\times\R^N)}\lesssim\varepsilon.
\end{equation}
We estimate the solution $u$ at time $T_1$:% Using Duhamel's representation, we write 
\begin{align}\label{duhamelexp1}
u(T_1)=e^{iT_1\Delta}u_0 + i \int_{0}^{T_1}e^{i(T_1-s)\Delta}(|u|^{p-1}u)(s)\,ds,
\end{align}
and for $t\in[0,T_1]$
\begin{align}\label{duhamelexp2}
e^{i(t-T_1)\Delta}u(T_1)=e^{it\Delta}u_0 + F_j(t);\quad j=1,2
\end{align}		
with
$$
F_j(t)=i\int_{I_j}^{}e^{i(t-s)\Delta}(|u|^{p-1}u)(s)\,ds\quad\text{and}\quad I_1=[T_1-\varepsilon^{-\alpha},T_1],\quad I_2=[0,T_1-\varepsilon^{-\alpha}]. 
$$
From \eqref{linearcomp} we note that the contribution from the linear component in \eqref{duhamelexp2} is small. For the second term, using Minkowski's inequality, dispersive estimate \eqref{dispest}, then considering characteristic function $\chi_{I_1}(t^{\prime})$ and  Lemma \ref{HLS}, we obtain for $t\in I_1$, where $I_1 = [T_1-\varepsilon^{-\alpha},T_1]$
\begin{align}\label{OM-change3}
\|\int_{T_1-\varepsilon^{-\alpha}}^{T_1}e^{i(t-s)\Delta}&(|u|^{p-1}u)(s)\,ds\|_{L^{\frac{2(N+2)}{N-2s}}_{t,x}([T_1,\infty)\times\R^N)}\lesssim\||u|^{p-1}u\|_{L_{t}^{2}L_x^{\frac{2N}{N+2(1-s)}}(I_1\times\R^N)}\\\notag
&\lesssim\|u\|^{p-1}_{L_{t}^{(N+2)(p-1)}L_x^{\frac{N(N+2)(p-1)}{2(N+1)}}(I_1\times\R^N)}\|u\|_{L_t^{\frac{2(N+2)}{N}}L_x^{\frac{2N(N+2)}{N^2-2s(N+2)}}(I_1\times\R^N)}\\\notag
&\lesssim\|u\|^{(p-1)-\frac{2}{N}}_{L_t^{\infty}L_x^{\frac{2N}{N-2s}}(I_1\times\R^N)}
\||\nabla|^su\|^{1+\frac{2}{N}}_{L_t^{\frac{2(N+2)}{N}}L_x^{\frac{2(N+2)}{N}}(I_1\times\R^N)}.
\end{align}		
Using \eqref{smallbnd1}, \eqref{localmass} and Lemma \ref{radsob}, we have
\begin{align*}
		&\|u\|_{L^{\infty}_tL^{\frac{2N}{N-2s}}_x(I_1\times\R^N)}=\|u\|_{L^{\infty}_tL^{\frac{2N}{N-2s}}_x(I_1\times B(0,R/2))}+\|u\|_{L^{\infty}_tL^{\frac{2N}{N-2s}}_x(I_1\times \R^N\backslash B(0,R/2))}\\
		&\lesssim\|\chi_Ru\|^{1-s}_{L_t^{\infty}L_x^2(I_1\times\R^N)}\|u\|^{s}_{L_t^{\infty}L_x^{\frac{2N}{N-2}}(I_1\times \R^N)}	+ \|(1-\chi_R)u\|^{2s/N}_{L_t^{\infty}L_x^{\infty}(I_1\times\R^N)}\|u\|^{(N-2s)/N}_{L_t^{\infty}L_x^{2}(I_1\times\R^N)}\\
		&\lesssim\varepsilon^{1-s}+R^{-\frac{2s}{N}},
\end{align*}
and choosing $R > 0$
%=R(\varepsilon)$  
 such that $R^{-\frac{2s}{N}}\ll\varepsilon^{1-s}$, we get
\begin{align}\label{bound1}
\|u \|_{L^{\infty}_tL^{\frac{2N}{N-2s}}_x(I_1\times\R^N)}\lesssim\varepsilon^{1-s}.
\end{align}
Therefore, we obtain
\begin{align}\label{bound2}
\|\int_{T_1-\varepsilon^{-\alpha}}^{T_1}e^{i(t-s)\Delta}(|u|^{p-1}u)(s)\,ds\|_{L_{t,x}^{\frac{2(N+2)}{N-2s}}([T_1,\infty)\times\R^N)}&
\lesssim\varepsilon^{(1-s)\left(\frac{N(p-1)-2}{N}\right)-\frac{\alpha}{2}}.
\end{align}
Take $0<\beta_1< (1-s) \left(\frac{(N(p-1)-2)}{N}\right)-\frac{\alpha}2 $, then  
\begin{align}\label{2ndterm}
\|\int_{T_1-\varepsilon^{-\alpha}}^{T_1}e^{i(t-s)\Delta}(|u|^{p-1}u)(s)\,ds\|_{L_{t,x}^{\frac{2(N+2)}{N-2s}}([T_1,\infty)\times\R^N)}\lesssim\varepsilon^{\beta_1}.
\end{align}
Recalling the definition of $F_2(u(t))$ from \eqref{duhamelexp2}, we split it via interpolation as
\begin{align*}
\|F_2(u)\|_{L_{t,x}^{\frac{2(N+2)}{N-2s}}([T_1,\infty)\times\R^N)}\lesssim\|F_2(u)\|^{1-s}_{L_{t,x}^{\frac{2(N+2)}{N}}([T_1,\infty)\times\R^N)}\|F_2(u)\|^{s}_{L_{t,x}^\frac{2(N+2)}{N-2}([T_1,\infty)\times\R^N)}.
\end{align*}
By the dispersive estimate \eqref{dispest} for $t\in[T_1,\infty)$, we bound
\begin{align}\label{Ng2}
\|F_2(u(t))\|_{L^{\frac{2(N+2)}{N-2}}_x(\R^N)}\lesssim\int_{0}^{T_1-\varepsilon^{-\alpha}}(t-s)^{-\frac{2N}{N+2}}\|u(s)\|^{p-1}_{L^{p_1}(\R^N)}\|u(s)\|_{L^{p_2}(\R^N)}ds,
\end{align}
where $p_1=\frac{2N}{N-2s}$ and $p_2=\frac{2N(N+2)}{(N+4)(N-2)}$. 
Observe that for $N> 2$, we have $H^1(\R^N)\hookrightarrow L^{p_{i}}(\R^N)$ for $i=1,2$, see also our Remark \ref{N3} about the restriction $N>2$, and thus, by \eqref{unibnd} we obtain
\begin{align}\label{dimrem1}
\|F_2(u(t))\|_{L^{\frac{2(N+2)}{N-2}}_x(\R^N)} 
& \lesssim {((t-T_1)+\varepsilon^{-\alpha})^{-\left(\frac{N-2}{N+2}\right)}},
\end{align}
which gives
\begin{align}\label{E:forbeta2}
\|F_2(u)\|_{L_{t,x}^\frac{2(N+2)}{N-2}([T_1,\infty)\times\R^N)}\lesssim
{\varepsilon^{\frac{\alpha(N-2)}{2(N+2)}}}.
\end{align}
Rewriting $F_2$ via Duhamel's formula applied on $[0, T_1-\eps^{-\alpha}))$, we get
$$
F_2(t)={e^{i(t-T_1+\varepsilon^{-\alpha})\Delta}u(T_1-\varepsilon^{-\alpha})- e^{it\Delta}u(0)},
$$ 
and using the homogeneous Strichartz estimate \eqref{stri1} and \eqref{unibnd}, yields
$$
\|F_2(u)\|_{L^\frac{2(N+2)}{N}_{t,x}([T_1,\infty),\R^N)}\lesssim \|u(T_1-\varepsilon^{-\alpha})\|_{L^2(\R^N)}+\|u(0)\|_{L^2(\R^N)}\lesssim 1.
$$
Take $0<\beta_2< \frac{\alpha(N-2)}{2(N+2)}$, then from \eqref{E:forbeta2}, we estimate $F_2$ as
\begin{align}\label{3rdterm}
\|F_2(u)\|_{L_{t,x}^{\frac{2(N+2)}{N-2s}}([T_1,+\infty)\times\R^N)}\lesssim\varepsilon^{\beta_2}.
\end{align}
Putting together \eqref{linearcomp}, \eqref{2ndterm}, and \eqref{3rdterm} gives
\begin{align}\label{bound}
\|e^{i(t-T_1)\Delta}u(T_1)\|_{L_{t,x}^{\frac{2(N+2)}{N-2s}}([T_1,+\infty)\times\R^N)}\lesssim\varepsilon^{\mu},
\end{align}
where $\mu \leq \min \{\beta_1,\beta_2\}>0$.
		
Note that on $[T_1,t]$, we have
\begin{align}\label{Duhamelexp2}
u(t)=e^{i(t-T_1)\Delta}u(T_1)-i\int_{T_1}^{t}e^{i(t-s)\Delta}|u|^{p-1}u(s)\,ds.
\end{align}
From \eqref{bound} and Strichartz estimate, we observe that
\begin{align*}
\|u \|_{L_{t,x}^{\frac{2(N+2)}{N-2s}}([T_1,+\infty)\times\R^N)}&\lesssim\varepsilon^{\mu}+ \|\nabla( |u|^{p-1}u) \|_{L_{t,x}^{\frac{2(N+2)}{N+4}}([T_1,+\infty)\times\R^N)}\\
&\lesssim \|u \|^{p-1}_{L_{t,x}^{\frac{2(N+2)}{N-2s}}([T_1,+\infty)\times\R^N)} 
\|\nabla u \|_{L_{t,x}^{\frac{2(N+2)}{N}}([T_1,+\infty)\times\R^N)}.
\end{align*}
Thus, from \eqref{stri1}, \eqref{unibnd}, \eqref{3rdterm} and using a standard continuity argument on the nonlinear flow, we observe that
\begin{equation}\label{scatbnd}
\|u\|_{L_{t,x}^{\frac{2(N+2)}{N-2s}}([T_1,+\infty)\times\R^N)}\lesssim\varepsilon^{\mu}.
\end{equation}
Now, we define
\begin{align}\label{finalstate}
u_+:=e^{-iT_1\Delta}u(T_1)-i\int_{T_1}^{+\infty}e^{-is\Delta}F(u(s))\,ds,
\end{align}
and since we have shown that $F(u)$ lies in $L_{t}^{\frac{2(N+2)}{N+4}}W_x^{1,\frac{2(N+2)}{N+4}}([T_1,+\infty)\times\R^N)$, this implies that $u_+\in H^1(\R^N)$. Then from \eqref{Duhamelexp2} and \eqref{finalstate}, we have
%\begin{align*}
%%&u(t)=e^{it\Delta}\left(u_+ + i\int_{T_1}^{+\infty}e^{-it'\Delta}F(u(t'))\,dt'\right)-i\int_{T_1}^{t}e^{i(t-t')\Delta}F(u(t'))\,dt'\\
%&
$$
u(t)-e^{it\Delta}u_+=i\int_{t}^{+\infty}e^{i(t-s)\Delta}F(u(s))\,ds
$$		
%\end{align*}
for all $t\geq T_1$, hence
\begin{align*}
\|u(t)-e^{it\Delta}u_+\|_{H^1}			&\lesssim\|(1+|\nabla|)F(u)\|_{S'(L^2;[t,+\infty))}\\
			&\lesssim \|u\|^{p-1}_{L_{t,x}^{\frac{2(N+2)}{N-2s}}([T_1,+\infty)\times\R^N)}\|(1+|\nabla|) u\|_{L_{t,x}^{\frac{2(N+2)}{N}}([T_1,+\infty)\times\R^N)}.
\end{align*} 
By \eqref{unibnd} and \eqref{scatbnd} we observe that the Strichartz norm on $[T_1,+\infty)$ for the above expression is bounded, therefore, the tail has to vanish as $t\rightarrow+\infty$. Hence,
$$
\lim\limits_{t\rightarrow+\infty}||u(t)-e^{it\Delta}u_+||_{H^1(\R^N)}=0.
$$
\end{proof}

We now prove a scattering criterion for the radial solutions to the gHartree equation \eqref{gH}.

\begin{lemma}[Scattering criterion for gHartree]\label{Scat_crit_gH}
Consider $0<s<1$. Suppose $u$ is a radial solution to \eqref{gH} satisfying
%\begin{align*}
$$	
\sup_{t\in[0,\infty)}\|v(t)\|_{H^1(\R^N)}\leq E.
$$	%\end{align*}
If there exist constants $\eps>0$ and $R >0$, depending only on %the energy 
$E$, such that
\begin{equation}\label{masscond_gH}
\liminf_{t\rightarrow\infty}\int_{|x|\leq R}^{}|v(x,t)|^2\,dx\leq\varepsilon,
\end{equation}
then $v(t)$ scatters forward in time.
\end{lemma}
\begin{remark}
	As in the NLS case, the proof below uses a symmetric pair (same as in the NLS \eqref{OM-change}) and thus, works for $0<s<1/2$. For the full intercritical range $0<s<1$, the proof is given in Appendix \ref{OM3}.
\end{remark}
\begin{proof} The proof is similar to Lemma \ref{Scat_crit} except for the estimate for the following terms:
$$
F_1(t)=i\int_{T_1-\varepsilon^{-\alpha}}^{T_1}e^{i(t-s)\Delta}\left((|x|^{-(N-\gamma)}\ast|v|^p)|v|^{p-2}v\right)(s)\,ds
$$
and
$$
F_2(t)=i\int_{0}^{T_1-\varepsilon^{-\alpha}}e^{i(t-s)\Delta}\left((|x|^{-(N-\gamma)}\ast|v|^p)|v|^{p-2}v\right)(s)\,ds,
$$
where $t\in[0,T_1]$. To derive the estimate of $F_1(t)$, we proceed as follows (using the same argument as in Lemma \ref{Scat_crit})
\begin{align*}
\|F_1(v)\|_{L_{t,x}^{\frac{2(N+2)}{N-2s}}([T_1,\infty)\times\R^N)}&\lesssim \| (|x|^{-(N-\gamma)}\ast|v|^p)|v|^{p-2}v\|_{{L_t^{2}L_x^{\frac{2N}{N+2(1-s)}}}(I_1\times\R^N)}
\end{align*}
for $t\in I_1=[T_1-\varepsilon^{-\alpha},T_1]$.	Using the product rule, H\"older's together with the Hardy-Littlewood-Sobolev (Lemma \ref{HLS}) and Sobolev inequalities, we get
\begin{align}\label{F1}\notag
	\|(&|x|^{-(N-\gamma)}\ast|v|^p)|v|^{p-2}v\|_{L_t^{2}L_x^{\frac{2N}{N+2(1-s)}}(I_1\times\R^N)}\\\notag
	\lesssim&\,\||x|^{-(N-\gamma)}\ast|v|^p\|_{L^{\frac{2(N+2)}{N+1}}_tL^{\frac{2N(N+2)}{N^2+2N+2-2s(N+2)-\gamma(N+2)}}_x(I_1\times\R^N)}\|v\|^{p-1}_{L^{2(N+2)(p-1)}_tL^{\frac{2N(N+2)(p-1)}{2(N+1)+\gamma(N+2)}}_x(I_1\times\R^N)}\\\notag
	\lesssim&\,\|v\|^p_{L^{\frac{2(N+2)p}{N+1}}_tL^{\frac{2N(N+2)p}{N^2+2N+2-2s(N+2)+\gamma(N+2)}}_x(I_1\times\R^N)}\|v\|^{p-1}_{L^{2(N+2)(p-1)}_tL^{\frac{2N(N+2)(p-1)}{2(N+1)+\gamma(N+2)}}_x(I_1\times\R^N)}\\\notag
	\lesssim&\,\| v\|_{L^{\frac{2(N+2)}{N}}_{t}L^{\frac{2N(N+2)}{N^2-2s(N+2)}}(I_1\times\R^N)}\|v\|^{2(p-1)}_{L^{2(N+2)(p-1)}_tL^{\frac{2N(N+2)(p-1)}{2(N+1)+\gamma(N+2)}}_x(I_1\times\R^N)}\\
	\lesssim&\,\|v\|^{2(p-1)}_{L^{2(N+2)(p-1)}_tL^{\frac{2N(N+2)(p-1)}{2(N+1)+\gamma(N+2)}}_x(I_1\times\R^N)}\||\nabla|^s v\|_{L^{\frac{2(N+2)}{N}}_{t,x}(I_1\times\R^N)}.
\end{align}

The first term in \eqref{F1} can be estimated as
	\begin{align*}
	\|v\|^{2(p-1)}_{L^{2(N+2)(p-1)}_tL^{\frac{2N(N+2)(p-1)}{2(N+1)+\gamma(N+2)}}_x(I_1\times\R^N)}\lesssim\|v\|^{2(p-1)-\frac{2}{N}}_{L_t^{\infty}L_x^{\frac{2N}{N-2s}}(I_1\times\R^N)}\||\nabla|^sv\|^{\frac{2}{N}}_{L_t^{\frac{2(N+2)}{N}}L_x^{\frac{2(N+2)}{N}}(I_1\times\R^N)}.
	\end{align*}
Therefore, we obtain
$$
\|F_1(v)\|_{L_{t,x}^{\frac{2(N+2)}{N-2s}}(I_1\times\R^N)}\lesssim\varepsilon^{2(1-s)\left(\frac{N(p-1)-1}{N}\right)-\frac{\alpha}{2}}.
$$
	So, for $\alpha<4(1-s)\left(\frac{N(p-1)-1}{N}\right)$ (note that $\frac{N(p-1)-1}{N}>0\implies p>1+\frac{1}{N}$, which holds true since $s>0$), we have that 
\begin{align}\label{2ndtermgH}
\|F_1(v)\|_{L_{t,x}^{\frac{2(N+2)}{N-2s}}(I_1\times\R^N)}\lesssim\varepsilon^{\beta_3},
\end{align}
where $\beta_3>0$ chosen in the same way as in the previous lemma. To derive the estimate for $F_2(t)$ we argue in a similar fashion as in Lemma \ref{Scat_crit}. By interpolation,
\begin{align*}
\|F_2(v)\|_{L_{t,x}^{\frac{2(N+2)}{N-2s}}([T_1,\infty)\times\R^N)}\lesssim\|F_2(v)\|^{1-s}_{L_{t,x}^{\frac{2(N+2)}{N}}([T_1,\infty)\times\R^N)}\|F_2(v)\|^{s}_{L_{t,x}^\frac{2(N+2)}{N-2}([T_1,\infty)\times\R^N)}.
\end{align*}
We bound the last term above for $t\in[T_1,\infty)$ using the dispersive estimate \eqref{dispest} as
$$
\|F_2(v(t))\|_{L_x^{\frac{2(N+2)}{N-2}}(\R^N)}\lesssim\int_{0}^{T_1-\varepsilon^{\alpha}}(t-s)^{-\frac{2N}{N+2}}\|v\|_{L^{p_2}(\R^N)}^p\|v\|^{p-1}_{L^{p_1}(\R^N)}ds,
$$
where $p_1=\frac{2N}{N-2s}$ and $p_2=\frac{2N(N+2)p}{N(N+6)+(N+2)(\gamma-2)}$. Observe that for $N>2$ one has the embedding $H^1(\R^N)\hookrightarrow L^{p_{i}}(\R^N)$ with $i=1,2$,  
%$2<\frac{2N(N+2)p}{N(N+6)+(N+2)(\gamma-2)}<\frac{2N}{N-2}$ (ensures $-c_1\leq s\leq 1+c_2$) and $2<\frac{2N}{(N-2s)}<\frac{2N}{N-2}$ (i.e., $0<s<1$), 
and thus, by \eqref{unibnd}
\begin{align}\label{dimrem2}
\|F_2(v(t))\|_{L_x^{\frac{2(N+2)}{N-2}}(\R^N)}\lesssim ((t-T_1)+\varepsilon^{-\alpha})^{-\left(\frac{N-2}{N+2}\right)}.
\end{align}
Next, note that 
$$
F_2(t)=e^{i(t-T_1+\varepsilon^{-\alpha})\Delta} v(T_1-\varepsilon^{-\alpha}) - e^{it\Delta} v(0)),
$$
and using the homogeneous Strichartz estimate \eqref{stri1} and \eqref{unibnd}
$$
\|F_2(v)\|_{L^\frac{2(N+2)}{N}_{t,x}([T_1,\infty),\R^N}\lesssim \|v(T_1-\varepsilon^{-\alpha})\|_{L^2(\R^N)}+\|v(0)\|_{L^2(\R^N)}\lesssim 1.
$$
Similarly, to the step \eqref{3rdterm}, we take $\beta_4>0$ so that
\begin{align}\label{3rdtermgH}
\|F_2(v)\|_{L_{t,x}^{\frac{2(N+2)}{N-2s}}([T_1,+\infty)\times\R^N)}\lesssim\varepsilon^{\beta_4}.
\end{align}
Putting together \eqref{linearcomp}, \eqref{2ndtermgH}, and \eqref{3rdtermgH} gives
\begin{align}\label{boundgH}
\|e^{i(t-T_1)\Delta}v(T_1)\|_{L_{t,x}^{\frac{2(N+2)}{N-2s}}([T_1,+\infty)\times\R^N)}\lesssim\varepsilon^{\nu}
\end{align}
for $0<\nu \leq \min \{\beta_3, \beta_4\}$. For the bound on the nonlinear solution we again consider Duhamel's formula
\begin{align*}
v(t)=e^{i(t-T_1)\Delta}v(T_1)-i\int_{T_1}^{t}e^{i(t-s)\Delta}F(v(s))\,ds,
\end{align*}
where $F(v) = (|x|^{-(N-\gamma)}\ast|v|^p)|v|^{p-2}v$. Taking $L_{t,x}^{\frac{2(N+2)}{N-2s}}([T_1,+\infty)\times\R^N)$ norms, we observe from the linear evolution bound and inhomogeneous Strichartz estimate \eqref{stri4} that
$$
\|v\|_{L_{t,x}^{\frac{2(N+2)}{N-2s}}([T_1,\infty)\times\R^N)}\lesssim\varepsilon^{\nu}+\|\nabla F(v)\|_{L_{t}^{\frac{2(N+2)}{N+2(\gamma+2)}}L_x^{\frac{2N(N+2)}{N^2+4(N-\gamma)}}([T_1,\infty)\times\R^N)}.
$$
By H\"older's, product rule and Lemma \ref{HLS}
\begin{align*}
\|\nabla\big((|x|^{-(N-\gamma)}&\ast|v|^p)|v|^{p-2}v\big)\|_{L_{t}^\frac{2(N+2)}{N+2(\gamma+2)}L_x^{\frac{2N(N+2)}{N^2+4(N-\gamma)}}([T_1,\infty)\times\R^N)}\\
&
\lesssim\,\|v\|^{2(p-1)}_{L^{\frac{2(N+2)}{N-2s}}_{t,x}([T_1,\infty)\times\R^N)}\|\nabla v\|_{L^{\frac{2(N+2)}{N}}_{t,x}([T_1,\infty)\times\R^N)}.
\end{align*}
The remaining proof follows similar reasoning as in Lemma \ref{Scat_crit} and we obtain the $H^1$ scattering of $v(t)$.
\end{proof}

\section{Variational analysis}\label{var}
The variational analysis for NLS is well known, see \cite{HR07}, \cite{HR08}, \cite{G}, however, for the completeness and comparison with gHartree case we include it here. (Note that in NLS we can use the ground state $Q$, though in gHartree we use the sharp constant $C_{GNC}$ as $Q$ in general may not be unique. See more on this in \cite{AKAR}.)
\subsection{The NLS equation}\label{variational}
The ground state solution $Q$ of \eqref{NLSQeq} optimizes the Gagliardo-Nirenberg inequality
$$
\|u\|^{p+1}_{L^{p+1}}\leq C_{GN}\|\nabla u\|^{\frac{N(p-1)}{2}}_{L^2(\R^N)}\|u\|^{2-\frac{(N-2)(p-1)}{2}}_{L^2(\R^N)},
$$
with 
$$
C_{GN} = \frac{2(p+1)}{2N-(N-2)(p+1)} \left(\frac{N(p-1)}{2(p+1)-N(p-1)} \right)^{-N(p-1)/4} \|Q\|^{-(p-1)}.
$$
The Pohozaev identities for $Q$ yield
\begin{align}\label{poh1}
\|Q\|^{1-s}_{L^2(\R^N)}\|\nabla Q\|_{L^2(\R^N)}^s=\left(\frac{2(p+1)}{N(p-1)C_{GN}}\right)^{\frac{1}{p-1}}	
\end{align}
and
\begin{align}\label{poh2}
	M(Q)^{1-s}E(Q)^s=\left(\frac{s}{N}\right)^s\left(\|Q\|^{1-s}_{L^2(\R^N)}\|\nabla Q\|_{L^2(\R^N)}^s\right)^2.
\end{align}
For the the rest of this subsection, we suppose $u_0\in H^1(\R^N)$ and $u(t)$ solves the NLS equation \eqref{NLS}.
\begin{lemma}\label{Coercivity I}
	If $M[u_0]^{\frac{1-s}{s}}E[u_0]<(1-\delta)M[Q]^{\frac{1-s}{s}}E[Q]$ and $\|u_0\|^{\frac{1-s}{s}}_{L^2(\R^N)}\|\nabla u_0\|_{L^2(\R^N)}\leq \|Q\|^{\frac{1-s}{s}}_{L^2(\R^N)}\|\nabla Q\|_{L^2(\R^N)}$, then there exists $\delta_1=\delta_1(\delta)>0$ so that
	$$
	\|u(t)\|^{\frac{1-s}{s}}_{L^2(\R^N)}\|\nabla u(t)\|_{L^2(\R^N)}<(1-\delta_1)\|Q\|^{\frac{1-s}{s}}_{L^2(\R^N)}\|\nabla Q\|_{L^2(\R^N)}
	$$
	for all $t\in I$, where $u\,:\,I\times\R^N\rightarrow\C$ is the maximal lifespan solution to \eqref{NLS}. In particular, $I=\R$ and $u$ is uniformly bounded in $H^1(\R^N)$.
\end{lemma}

\begin{proof} While this is a simple and well-known proof, it's the core of the dichotomy and we include it for completeness. 
By the mass and energy conservation along with Gagliardo-Nirenberg inequality
	\begin{align*}
		(1-\delta)&M[Q]^{\frac{1-s}{s}}E[Q]\geq M[u_0]^{\frac{1-s}{s}}E[u_0] \\
		&\geq \frac{1}{2}\left(\|u(t)\|^{\frac{1-s}{s}}_{L^2}\|\nabla u(t)\|_{L^2}\right)^2 - \frac{1}{p+1}C_{GN}\left(\|u(t)\|^{\frac{1-s}{s}}_{L^2}\|\nabla u(t)\|_{L^2}\right)^{s(p-1)+2}.
\end{align*}
Using \eqref{poh1} and \eqref{poh2}, the above estimate becomes
\begin{align*}
		1-\delta \geq \frac{N}{2s}\left(\frac{\|u(t)\|^{\frac{1-s}{s}}_{L^2}\|\nabla u(t)\|_{L^2}}{\|Q\|^{\frac{1-s}{s}}_{L^2}\|\nabla Q\|_{L^2}}\right)^2 - \frac{2}{s(p-1)}\left(\frac{\|u(t)\|^{\frac{1-s}{s}}_{L^2}\|\nabla u(t)\|_{L^2}}{\|Q\|^{\frac{1-s}{s}}_{L^2}\|\nabla Q\|_{L^2}}\right)^{s(p-1)+2}.
\end{align*}
Define $f(x)=\frac{N}{2s}x^2-\frac{2}{s(p-1)}x^{s(p-1)+2}$. Since $s>0$, we always have $\text{deg}(f)>2$. Therefore,
\begin{align*}
f'(x)=\frac{N}{s}x-\frac{2(s(p-1)+2)}{s(p-1)}x^{s(p-1)+1}=\frac{N}{s}\left(1-x^{s(p-1)}\right)x,
\end{align*}
which implies that $f'(x)=0$ when $x_0=0$ and $x_1=1$. Observe that
\begin{align}\label{coer1}
	f(x)\leq 1-\delta <1=f(x_1).
\end{align}
If initially we have
$$
\|u_0\|^{\frac{1-s}{s}}_{L^2(\R^N)}\|\nabla u_0\|_{L^2(\R^N)}\leq \|Q\|^{\frac{1-s}{s}}_{L^2(\R^N)}\|\nabla Q\|_{L^2(\R^N)},
$$
then by \eqref{coer1} and the continuity of $\|\nabla u(t)\|_{L^2(\R^N)}$ in $t$, we conclude that
$$
\|u_0\|^{\frac{1-s}{s}}_{L^2(\R^N)} \|\nabla u(t)\|_{L^2(\R^N)}
\leq \|Q\|^{\frac{1-s}{s}}_{L^2(\R^N)} \|\nabla Q \|_{L^2(\R^N)},\,\,\,\text{i.e.,}\,\,\,x<x_1
$$
for all time $t\in I$. Furthermore, since the $L^2$-norm is conserved and recalling the $H^1$ blowup criterion implies that we have global existence and the $\dot{H}^1$-norm is uniformly bounded.
\end{proof}
Based on the result of Lemma \ref{Coercivity I}, we have that $u$ is global and uniformly bounded in $H^1$, moreover, there exists $\delta>0$ such that
\begin{align}\label{coercivity}
\sup_{t\in\R } \|u\|^{\frac{1-s}{s}}_{L^2(\R^N)} \|\nabla u \|_{L^2(\R^N)}<(1-\delta) \|Q \|^{\frac{1-s}{s}}_{L^2(\R^N)} \|\nabla Q\|_{L^2(\R^N)}.
\end{align}
To prove Theorem \ref{main1} we use a virial weight in a ball around the origin of sufficiently large radius together with the coercivity to obtain a suitable lower bound. First we need \eqref{coercivity} on balls of sufficiently large radii so that we can have a necessary coercivity. Define   $\chi_R(x)=\chi\left(\frac{x}{R}\right)$ for $R>0$,  where $\chi(x)\in C^{\infty}_c(\R^N)$ is a smooth cutoff function on $\{|x|\leq 1\}$ with $\chi(x)=1$ for $|x|\leq \frac{1}{2}$.

\begin{lemma}\label{coercivityball}
There exists $R_0=R_0(\delta,M(u),Q)>0$ sufficiently large so that
\begin{align}\label{cball1}
\sup_{t\in\R}\|\chi_Ru(t)\|^{1-s}_{L^2} \|\chi_Ru(t) \|_{\dot{H}^1}^s<(1-\delta) \|Q \|_{L^2}^{1-s}\|\nabla Q\|_{L^2}^s.
\end{align}	
In particular, there exists $\delta_1=\delta_1(\delta)>0$ so that
\begin{align}\label{cball2}
\|\chi_Ru(t)\|_{\dot{H}^1}^2-\frac{N(p-1)}{2(p+1)}\|\chi_Ru(t)\|_{L^{p+1}}^{p+1}\geq \delta_1 \|\chi_Ru(t)\|_{L^{p+1}}^{p+1}
\end{align}
uniformly for $t\in\R$.
\end{lemma}
\begin{proof} For $u\in \dot{H^1}$, we write
$$
\|\nabla u \|_{L^2(\R^N)}^2-\frac{N(p-1)}{2(p+1)}
\|u\|_{L^{p+1}(\R^N)}^{p+1} =\frac{N(p-1)}{2}E[u]-\frac{s(p-1)}{2} \|\nabla u \|_{L^2(\R^N)}.
$$
By the Gagliardo-Nirenberg inequality and \eqref{poh1},
\begin{align*}
E[u]&\geq \frac{1}{2} \|\nabla u \|^2_{L^2(\R^N)}\left(1-\frac{2C_{GN}}{p+1} \left(\|u\|^{1-s}_{L^2(\R^N)} \|\nabla u \|^s_{L^2(\R^N)}\right)^{p-1}\right)\\
&\geq \frac{1}{2} \|\nabla u \|^2_{L^2(\R^N)}\left(1-\frac{2C_{GN}}{p+1}(1-\delta)\left(\|Q\|^{1-s}_{L^2(\R^N)} \|\nabla Q \|^s_{L^2(\R^N)}\right)^{p-1}\right)\\
&=\left(\frac{s}{N}+\frac{2\delta}{N(p-1)}\right) \|\nabla u \|^2_{L^2(\R^N)}.
\end{align*}
Therefore,
$$
\|\nabla u\|_{L^2(\R^N)}^2-\frac{N(p-1)}{2(p+1)} \|u \|_{L^{p+1}(\R^N)}^{p+1}\geq \delta \|\nabla u \|^2_{L^2(\R^N)},
$$
which implies that
$$
\|\nabla u\|_{L^2(\R^N)}^2-\frac{N(p-1)}{2(p+1)} \|u\|_{L^{p+1}(\R^N)}^{p+1}\geq \frac{N(p-1)}{2(p+1)}\,\frac{\delta}{1-\delta}\, \|u \|^{p+1}_{L^{p+1}(\R^N)},
$$
choosing $\displaystyle\delta_1=\frac{N(p-1)}{2(p+1)}\,\frac{\delta}{1-\delta}$ gives \eqref{cball2}. To finish the proof, we need to verify \eqref{cball1}. Observe that
$$
\|\chi_Ru(t)\|_{L^2_x}\leq \|u(t)\|_{L^2_x}
$$
uniformly for $t\in\R$. Therefore, it is sufficient to consider the $\dot{H}^1$ term. We compute
\begin{align*}
	\int |\nabla (\chi_R u)|^2\,dx&= \int\chi_R^2|\nabla u|^2\,dx + \int |u|^2\nabla\chi_R\nabla\chi_R+\int 2\Re(\bar{u}\nabla u)\chi_R\nabla\chi_R\\
	&=\int\chi_R^2|\nabla u|^2\,dx-\int|u|^2\chi_R\Delta\chi_R
\end{align*}
	and use the following identity
	\begin{align}\label{identity}
	\int\chi_R^2|\nabla u|^2\,dx = \int |\nabla (\chi_R u)|^2\,dx + \int\chi_R\Delta(\chi_R)|u|^2\,dx
	\end{align}
	with the definition of $\chi$ to write
	\begin{align}\label{cball3}
	\|\nabla(\chi_Ru)\|^2_{L^2}\leq \|\nabla u\|_{L^2}^2 + \mathcal{O}\left(\frac{1}{R^2}M[u]\right).
	\end{align}
	Choosing $R_0$ sufficiently large depending on $\delta,$ $M[u]$ and $Q$, we get that \eqref{cball3} holds for any $R>R_0$, and yields the desired estimate \eqref{cball1}.
\end{proof}

\subsection{The gHartree equation}\label{variational_gH} The Gagliardo-Nirenberg inequality of convolution type
\begin{equation}\label{gnc}
	P(v)\leq C_{GNC}\|\nabla v\|^{Np-(N+\gamma)}_{L^2(\R^N)}\|v\|^{N+\gamma-(N-2)p}_{L^2(\R^N)}
\end{equation}
has the sharp constant $C_{GNC}$, (see discussion in \cite{AKAR}, Section 4 and recall that $Q$ in gHartree may not be unique) with
\begin{align*}
	C_{GNC}=\frac{2p}{N(p-1)-\gamma}\left(\frac{N+\gamma-(N-2)p}{N(p-1)-\gamma}\right)^{\frac{N(p-1)-\gamma}{2}-1}\|Q\|_{L^2(\R^N)}^{-2(p-1)},
	\end{align*}
	where ground states $Q$ solve the equation \eqref{gHQeq}. From the value of sharp constant $C_{GNC}$ (and Pohozaev identities), we get
\begin{align}\label{poh1gH}
\|Q\|_{L^2(\R^N)}^{1-s}\|\nabla Q\|_{L^2}^s=\left(\frac{p\,(C_{GNC})^{-1}}{s(p-1)+1}\right)^{\frac{1}{2(p-1)}},
\end{align}
and
\begin{align}\label{poh2gH}
M[Q]^{1-s}E[Q]^s=\left(\frac{s(p-1)}{2s(p-1)+2}\right)^s\left(\|Q\|_{L^2(\R^N)}^{1-s}\|\nabla Q\|_{L^2(\R^N)}^s\right)^2.
\end{align} 
For the rest of this subsection, we assume $v_0\in H^1(\R^N)$ and $v(t)$ solves the gHartree equation \eqref{gH}.
\begin{lemma}\label{Coercivity IgH}
If $M[v_0]^{\frac{1-s}{s}}E[v_0]<(1-\delta)M[Q]^{\frac{1-s}{s}}E[Q]$ and $||v_0||^{\frac{1-s}{s}}_{L^2(\R^N)} \|\nabla v_0 \|_{L^2(\R^N)}\leq 
\|Q \|^{\frac{1-s}{s}}_{L^2(\R^N)} \|\nabla Q \|_{L^2(\R^N)}$, then there exists $\delta_1=\delta_1(\delta)>0$ so that
$$
\|v(t) \|^{\frac{1-s}{s}}_{L^2(\R^N)} \|\nabla v(t) \|_{L^2(\R^N)}<(1-\delta_1) \|Q\|^{\frac{1-s}{s}}_{L^2(\R^N)} \|\nabla Q \|_{L^2(\R^N)}
$$
for all $t\in I$, where $u\,:\,I\times\R^N\rightarrow\C$ is the maximal lifespan solution to \eqref{gH}. In particular, $I=\R$ and $u$ is uniformly bounded in $H^1(\R^N)$.
\end{lemma}

\begin{proof}
By the mass and energy conservation along with the Gagliardo-Nirenberg inequality
\begin{align*}
(1-\delta)&M[Q]^{\frac{1-s}{s}}E[Q]\geq M[v_0]^{\frac{1-s}{s}}E[v_0] \\
&\geq \frac{1}{2}\left(\|v(t)\|^{\frac{1-s}{s}}_{L^2}\|\nabla v(t)\|_{L^2}\right)^2 - \frac{1}{2p}C_{GNC}\left(\|v(t)\|^{\frac{1-s}{s}}_{L^2}\|\nabla v(t)\|_{L^2}\right)^{2s(p-1)+2},
\end{align*}
Using \eqref{poh1gH} and \eqref{poh2gH}, the above estimate becomes
\begin{align*}
1-\delta \geq \frac{s(p-1)+1}{s(p-1)}\left(\frac{\|v(t)\|^{\frac{1-s}{s}}_{L^2} 
\|\nabla v(t) \|_{L^2}}{||Q||^{\frac{1-s}{s}}_{L^2} \|\nabla Q \|_{L^2}}\right)^2 - \frac{1}{s(p-1)}\left(\frac{\|v(t) \|^{\frac{1-s}{s}}_{L^2} \|\nabla v(t) \|_{L^2}}{\|Q \|^{\frac{1-s}{s}}_{L^2} \|\nabla Q \|_{L^2}}\right)^{2s(p-1)+2}.
\end{align*}
Define $f(x)=\frac{s(p-1)+1}{s(p-1)}x^2-\frac{1}{s(p-1)}x^{2s(p-1)+2}$, since $s>0$, thus, $\text{deg}(f)>2$. Therefore,
\begin{align*}
f'(x)=\frac{2s(p-1)+2}{s(p-1)}x-\frac{2s(p-1)+2}{s(p-1)}x^{2s(p-1)+1}=\frac{2s(p-1)+2}{s(p-1)}\left(1-x^{2s(p-1)}\right)x,
\end{align*}
which implies that $f'(x)=0$ when $x_0=0$ and $x_1=1$. Observe that
\begin{align}\label{coer1gH}
f(x)\leq 1-\delta <1=f(x_1).
\end{align}
If initially we have
$$
\|v_0||^{\frac{1-s}{s}}_{L^2(\R^N)} \|\nabla v_0 \|_{L^2(\R^N)}
\leq \|Q \|^{\frac{1-s}{s}}_{L^2(\R^N)} \|\nabla Q \|_{L^2(\R^N)},
$$
then by \eqref{coer1gH} and the continuity of $\|\nabla v(t) \|_{L^2(\R^N)}$ in $t$, we conclude that
$$
\|v_0 \|^{\frac{1-s}{s}}_{L^2(\R^N)} \|\nabla v(t) \|_{L^2(\R^N)} \leq 
\|Q\|^{\frac{1-s}{s}}_{L^2(\R^N)} \|\nabla Q \|_{L^2(\R^N)},\,\,\,\text{i.e.,}\,\,\,x<x_1
$$
for all time $t\in I$. Since the $L^2$ norm of the gradient is bounded, we get global existence, completing the proof.
\end{proof}

Now we prove the coercivity estimate on balls of large radii for gHartree. The following lemma differs from the standard approach as we lower bound the virial not by the potential term but by the $L^{\frac{2Np}{N+\gamma}}(\R^N)$ norm.
\begin{lemma}\label{coercivityballgH_new}
	There exists $R_0=R_0(\delta,M(v),Q)>0$ suffiiently large that for any $R>R_0$
	\begin{align}\label{gHcball1}
		\sup_{t\in\R}\|\chi_Rv(t)\|^{1-s}_{L^2}\|\chi_Rv(t)\|_{\dot{H}^1}^s<(1-\delta)\|Q\|_{L^2}^{1-s}\|\nabla Q\|_{L^2}^s.
	\end{align}	
	In particular, there exists $\tilde{\delta}=\tilde{\delta}(\delta)>0$ so that
	\begin{align}\label{gHcball2}
		\|\chi_Rv(t)\|_{\dot{H}^1}^2-\frac{s(p-1)+1}{p}P[\chi_Rv(t)]\geq \tilde{\delta}\,\|\chi_Rv(t)\|^2_{L^{\frac{2Np}{N+\gamma}}}
	\end{align}
	uniformly for $t\in\R$.
\end{lemma}
\begin{proof}
	We write
	$$
	\|\nabla v\|_{L^2(\R^N)}^2-\frac{s(p-1)+1}{p}P(v)=(2s(p-1)+2)E[v]-s(p-1)\|\nabla v\|^2_{L^2(\R^N)}.
	$$
	By Gagliardo-Nirenberg inequality and \eqref{poh1gH},
	\begin{align*}
		E[v]&\geq \frac{1}{2}\|\nabla v\|^2_{L^2(\R^N)}\left(1-\frac{C_{GNC}}{p}\left(\|v\|^{1-s}_{L^2(\R^N)}\|\nabla v\|^s_{L^2(\R^N)}\right)^{2(p-1)}\right)\\
		&\geq \frac{1}{2}\|\nabla v\|^2_{L^2(\R^N)}\left(1-\frac{C_{GNC}}{p}(1-\delta)\left(\|Q\|^{1-s}_{L^2(\R^N)}\|\nabla Q\|^s_{L^2(\R^N)}\right)^{2(p-1)}\right)\\
		&=\left(\frac{s(p-1)}{2s(p-1)+2}+\frac{\delta}{2s(p-1)+2}\right)\|\nabla v\|^2_{L^2(\R^N)}.
	\end{align*}
	Therefore,
	$$
	\|\nabla v\|_{L^2(\R^N)}^2-\frac{s(p-1)+1}{p}P(v)\geq \delta \|\nabla v\|^2_{L^2(\R^N)}.
	$$
	The assumption $p<\frac{N+\gamma}{N-2}$ (since $s<1$) implies that $\frac{2Np}{N+\gamma}<\frac{2N}{N-2}$ and thus, Sobolev embedding gives that $\|v\|_{L^\frac{2Np}{N+\gamma}}\lesssim \|\nabla v\|_{L^2}$. Hence, we obtain
	$$
	\|\nabla v\|_{L^2(\R^N)}^2-\frac{s(p-1)+1}{p}P(v)\geq \tilde{\delta}\, \|v\|^2_{L^{\frac{2Np}{N+\gamma}}}.
	$$
	To verify \eqref{gHcball1} we follow the similar argument as in Lemma \ref{coercivityball} for \eqref{cball1} which concludes the proof. 
\end{proof}

\section{Proof of Theorem \ref{main1}}\label{mainNLS}
Suppose $u$ is a solution of \eqref{NLS} satisfying Theorem \ref{main1}. 
Next, we recall the Morawetz identity.
\begin{lemma}[Morawetz identity, NLS]\label{morawetz_id}
	Let $a\,:\,\R^N\rightarrow\R$ be a smooth weight. Define
	$$
	\mathcal{M}(t)=2\Im\int\bar{u}\,\nabla u\cdot\nabla a\,dx.
	$$
	Then
	$$
	\frac{d}{dt}\mathcal{M}(t)=\int-\frac{2(p-1)}{p+1}|u|^{p+1}\Delta a +|u|^2(-\Delta\Delta a)+4\Re a_{jk}\bar{u}_ju_k\,dx,
	$$
	where subscripts denote partial derivatives and repeated indices are summed.
\end{lemma}
We take $a$ to be a radial function satisfying
$$
a(x)=\begin{cases}
|x|^2\quad &|x|<\frac{R}{2}\\
R|x|\quad &|x|>R.
\end{cases}
$$
In the intermediate region $\frac{R}{2}<|x|\leq R$, we impose that
$$
\partial_ra> 0,\quad \partial^2_ra\geq 0,\quad |\partial^{\alpha}a(x)|\lesssim_{\alpha} R|x|^{-\alpha +1}\quad \text{for}\,\,|\alpha|\geq 1.
$$
Here, $\partial_r$ denotes the radial derivative, i.e., $\partial_ra=\nabla a\cdot \frac{x}{|x|}.$ Under these conditions, the matrix $a_{jk}$ is non-negative. Note that for $|x|\leq \frac{R}{2}$, we have
$$
a_{jk}=2\delta_{jk},\quad\Delta a=2N,\quad \Delta^2a=0,
$$
while for $|x|>R$, we have
$$
a_{jk}=\frac{R}{|x|}\left[\delta_{jk}-\frac{x_j}{|x|}\frac{x_k}{|x|}\right],\quad \Delta a=\frac{(N-1)R}{|x|},\quad \Delta^2a=0.
$$

\begin{proposition}[Morawetz estimate, NLS]\label{morawetz_est}
Let $T>0$. For $R\equiv R(\delta,M(u),Q)$ sufficiently large,
$$
\frac{1}{T}\int_{0}^{T}\int_{|x|\leq R}^{}|u(x,t)|^{p+1}\,dx\,dt\lesssim_{u,\delta}\begin{cases}
\frac{R}{T}+\frac{1}{R^{\frac{(N-1)(p-1)}{2}}},\quad &\text{if}\,\,s<\frac{1}{2}\\
\frac{R}{T}+\frac{1}{R^2},\quad &\text{if}\,\, s\geq \frac{1}{2}
\end{cases}.
$$
\end{proposition}

\begin{proof}
Note that by Cauchy-Schwarz, the uniform $H^1$ bound for $u$ and the choice of the weight $a(x)$, we have that
$
\sup_{t\in\R}|\mathcal{M}(t)|\lesssim_u R.
$
We compute
\begin{align*}
\frac{d}{dt}M(t)=&\,\,8\int_{|x|\leq \frac{R}{2}}^{}|\nabla u|^2-\frac{N(p-1)}{2(p+1)}|u|^{p+1}\,dx\\
&+\int_{|x|>R}^{}-\frac{2(N-1)(p-1)}{p+1}\frac{R}{|x|}|u|^{p+1}+\frac{4R}{|x|}|\blacktriangledown  u|^2\,dx\\
&+\int_{\frac{R}{2}<|x|\leq R}^{}4\Re a_{jk}\bar{u}_ju_k-\frac{2(p-1)}{p+1}\frac{R}{|x|}|u|^{p+1}-\frac{R}{|x|^3}|u|^2\,dx,
\end{align*}
where $\blacktriangledown$ denotes the angular part of the derivative. In fact, since $u$ is radial, this term is zero. We define $\chi_R:=\chi\left(\frac{x}{R}\right)$ for $R>0$ and write
\begin{align*}
\int_{|x|\leq \frac{R}{2} }^{}\chi_R^2|\nabla u|^2=\int_{|x|\leq\frac{R}{2}}^{}|\nabla(\chi_R u)|^2+\int_{|x|\leq \frac{R}{2}}^{}\chi_R\Delta(\chi_R) |u|^2.
\end{align*}
Then the Morawetz identity can be estimated as
\begin{align}\notag
\frac{d}{dt}\mathcal{M}(t)\,\,\geq&\,\, 8\int_{|x|\leq \frac{R}{2}}^{}|\nabla(\chi_R u)|^2- \frac{4N(p-1)}{p+1}\int_{|x|\leq \frac{R}{2}}^{}|\chi_Ru|^{p+1} -c_1\int_{|x|>\frac{R}{2}}^{}\frac{R}{|x|}|u|^{p+1}\\\label{NLSmor1}
    	&+ 8\int_{|x|\leq \frac{R}{2}}^{}\chi_R\Delta(\chi_R) |u|^2-\int_{\frac{R}{2}<|x|\leq R}^{}\frac{R}{|x|^3}|u|^2\,dx\\\label{NLSmor2}
    	\geq& \,\,8\int\delta_1|\chi_Ru(x,t)|^{p+1}\,dx-c_1\int_{|x|>\frac{R}{2}}^{}|u(x,t)|^{p+1}-\frac{c_2}{R^2}M(u),
\end{align}
where in the inequality \eqref{NLSmor2} we have used  Lemma \ref{coercivityball} and the fact that for a fixed radius $R$ and mass $M(u)$ the terms in \eqref{NLSmor1} is a constant multiple of $\frac{1}{R^2}$ and  $M(u)$. Next we apply the fundamental theorem of calculus on the interval $[0,T]$ and rearrange terms to obtain
\begin{align*}
\int_{0}^{T}\int8\delta_1|\chi_Ru(x,t)|^{p+1}&\,dx\,dt\lesssim\sup_{t\in[0,T]}|\mathcal{M}(t)|+\int_{0}^{T}\int_{|x|>R}^{}|u(x,t)|^{p+1}\,dx\,dt+\frac{T}{R^2}M(u).	
\end{align*}
By Lemma \ref{radsob} (radial Sobolev embedding), we have
\begin{align*}
\int_{|x|>R}^{}|u(x,t)|^{p+1}\,dx&\lesssim\frac{1}{R^{\frac{(N-1)(p-1)}{2}}}\int|x|^{\frac{(N-1)(p-1)}{2}}|u(x,t)|^{p-1}\,|u(x,t)|^2\,dx\\
	&\lesssim\frac{1}{R^{\frac{(N-1)(p-1)}{2}}}	||u||^{p-1}_{L_t^{\infty}\dot{H}^1_x}M(u)\\
	&\lesssim\frac{1}{R^{\frac{(N-1)(p-1)}{2}}}.
\end{align*}
Therefore, we deduce that
$$
\frac{1}{T}\int_{0}^{T}\int_{|x|\leq R}^{}|u(x,t)|^{p+1}\,dx\,dt\lesssim_{u,\delta}\frac{R}{T}+\frac{1}{R^{\frac{(N-1)(p-1)}{2}}}+\frac{1}{R^2}.
$$
Observe that
$$
\begin{cases}
\frac{(N-1)(p-1)}{2}\leq 2 \quad & \text{for}\,\,s\leq\frac{1}{2},\\
\frac{(N-1)(p-1)}{2}>2 \quad & \text{for}\,\,s>\frac{1}{2}.
\end{cases}
$$
Thus,
$$
\frac{1}{T}\int_{0}^{T}\int_{|x|\leq R}^{}|u(x,t)|^{p+1}\,dx\,dt\lesssim_{u,\delta}\begin{cases}
\frac{R}{T}+\frac{1}{R^{\frac{(N-1)(p-1)}{2}}},\quad &\text{if}\,\,s<\frac{1}{2}\\
\frac{R}{T}+\frac{1}{R^2},\quad &\text{if}\,\, s\geq \frac{1}{2}
\end{cases},
$$
as desired.
\end{proof}

Now we prove that the potential energy of $u$ escapes to spatial infinity as $t\rightarrow\infty$.

\begin{proposition}[Energy evacuation, NLS]\label{energyevac}
There exists a sequence of times $t_n\rightarrow\infty$ and a sequence of radii $R_n\rightarrow\infty$ such that
$$
\lim\limits_{n\rightarrow\infty}\int_{|x|\leq R_n}^{}
|u(x,t_n)|^{p+1}\,dx=0.	$$
\end{proposition}

\begin{proof}
For $T>0$ sufficiently large with $R=T^{\frac{2}{N(p-1)-(p-3)}}$ for $s<\frac{1}{2}$ and $R=T^{1/3}$ for $s\geq \frac{1}{2}$ (note that $R>R_0$) and Proposition \ref{morawetz_est}, we obtain for $s<\frac{1}{2}$ that
$$
\frac{1}{T}\int_{0}^{T}\int_{|x|\leq t^{\frac{2}{N(p-1)-(p-3)}}}^{}|u(x,t)|^{p+1}\,dx\,dt\lesssim T^{-\frac{(N-1)(p-1)}{N(p-1)-(p-3)}}
$$
and
$$
\frac{1}{T}\int_{0}^{T}\int_{|x|\leq t^{1/3}}^{}|u(x,t)|^{p+1}\,dx\,dt\lesssim T^{-2/3}
$$
for $s\geq \frac{1}{2}$. Since $u$ is global and $\|u\|_{L^{p+1}}$ is bounded, by mean value theorem there exist sequences $t_n\rightarrow\infty$ and $R_n\rightarrow\infty$ such that the energy evacuation happens on a sequence of balls as desired, i.e.,
	$$
	\lim\limits_{n\rightarrow\infty}\int_{|x|\leq R_n}^{}
	|u(x,t_n)|^{p+1}\,dx=0.
	$$
\end{proof}
\begin{proof}[Proof of Theorem \ref{main1}]
	We have (by Section \ref{variational}) that $u$ is global and uniformly bounded in $H^1$. Choose $\varepsilon$ and $R$ as in Lemma \ref{Scat_crit} with $t_n\rightarrow\infty$ and $R_n\rightarrow\infty$ as in Proposition \ref{energyevac}. Now taking $n$ large so that $R_n\geq R$, H\"older's inequality gives
	$$
	\int_{|x|\leq R}^{}|u(x,t_n)|^2\,dx\lesssim R^{\frac{N(p-1)}{p+1}}\left(\int_{|x|\leq R_n}^{}|u(x,t_n)|^{p+1}\,dx\right)^\frac{2}{p+1}\rightarrow 0 \quad \text{as}\quad n\rightarrow \infty.
	$$
	Therefore, Lemma \ref{Scat_crit} implies that $u$ scatters in $H^1(\R^N)$ forward in time.
\end{proof}
\section{Proof of Theorem \ref{main2}}\label{maingH}
We begin this section with the Morawetz identity in the gHartree case, which will be used to obtain Morawetz estimate.
\begin{lemma}[Morawetz identity, gH]\label{morawetz_idgH}
	Let $a\,:\,\R^N\rightarrow\R$ be the same smooth weight as described in Section \ref{mainNLS}. Define
	$$
	\mathcal{M}_{gH}(t)=2\Im\int\bar{v}\,\nabla v\cdot\nabla a\,dx.
	$$
	Then
	\begin{align*}
	\frac{d}{dt}\mathcal{M}_{gH}(t)=&\int-4\left(\frac{1}{2}-\frac{1}{p}\right)\left(|x|^{-(N-\gamma)}\ast|v|^p\right)|v|^p\Delta a +|v|^2(-\Delta\Delta a)+4\Re a_{jk}\bar{v}_jv_k\,dx\\
	&	-\frac{4(N-\gamma)}{p}\int\int\nabla a\frac{(x-y)|v(x)|^p|v(y)|^p}{|x-y|^{N-\gamma+2}}\,dxdy,
	\end{align*}
	where subscripts denote partial derivatives and repeated indices are summed.
\end{lemma}
\begin{proposition}[Morawetz estimate, gH]\label{morawetz_est_gH}
	Let $T>0$. For $R\equiv R(\delta,M(v),Q)>0$ sufficiently large,
	$$
	\frac{1}{T}\int_{0}^{T}\left(\int_{|x|\leq R}^{}|v(x,t)|^{\frac{2Np}{N+\gamma}}\,dx\right)^{\frac{N+\gamma}{Np}}\,dt\lesssim_{v,\delta}\begin{cases}
	\frac{R}{T}+\frac{1}{R^{\frac{(N-1)(N(p-1)-\gamma)}{N}}}\quad &\text{if}\,\,s<\frac{1}{2}\\
	\frac{R}{T}+\frac{1}{R^2}\quad &\text{if}\,\, s\geq \frac{1}{2}.
	\end{cases}
	$$
\end{proposition}
\begin{proof}
	Note that by Cauchy-Schwarz, the uniform $H^1$ bound for $u$ and the choice of the weight $a(x)$, we have
	$$
	\sup_{t\in\R}|\mathcal{M}_{gH}(t)|\lesssim R.
	$$
	We recall (from Lemma \ref{morawetz_idgH})
	\begin{align}\label{1}
		\frac{d}{dt}\mathcal{M}_{gH}(t)=& \int|v|^2(-\Delta\Delta a)+4\Re a_{jk}\bar{v}_jv_k\,dx\\\label{2}
		&-4\int\Delta a\left(\frac{1}{2}-\frac{1}{p}\right)\left(|x|^{-(N-\gamma)}\ast|v|^p\right)|v|^p\\\label{3}
		&	-\frac{4(N-\gamma)}{p}\int\int\nabla a\frac{(x-y)|v(x)|^p|v(y)|^p}{|x-y|^{N-\gamma+2}}\,dxdy.
	\end{align}
	For $|x|\leq \frac{R}{2}$, the above expression reduces to 
	\begin{align}\label{morterm1gH}
		8\int_{|x|\leq \frac{R}{2}}^{}|\nabla v|^2-\int_{|x|\leq \frac{R}{2}}^{}\frac{4(N(p-1)-\gamma)}{p}\left(|x|^{-(N-\gamma)}\ast|v|^p\right)|v|^p\,dx.
	\end{align}
	In the region $\frac{R}{2}<|x|\leq R$, \eqref{1} yields
	\begin{align}\label{morterm2gH}
		\int_{\frac{R}{2}<|x|\leq R}^{}4\Re a_{jk}\bar{v}_jv_k+\mathcal{O}\left(\int_{\frac{R}{2}<|x|\leq R}^{}\frac{R}{|x|^3}|v|^2\right)
	\end{align}
	and in $|x|>R$, it gives
	\begin{align}\label{morterm3gH}
	\int_{|x|>R}^{}\frac{4R}{|x|}|\blacktriangledown  v|^2\,dx=0,
	\end{align}
	where $\blacktriangledown$ denotes the angular part of the derivative, which we drop, since $v$ is radial. \\
	
	\noindent
	In the region $\frac{R}{2}<|x|\leq R$, \eqref{2} yields
	\begin{align}\label{morterm4}
		-\int_{\frac{R}{2}<|x|\leq R}^{}4\left(\frac{1}{2}-\frac{1}{p}\right)\frac{R}{|x|}\left(|x|^{-(N-\gamma)}\ast|v|^p\right)|v|^p\,dx
	\end{align}
	and in $|x|>R$, it gives
	\begin{align}\label{morterm5}
		-\int_{|x|>R}^{}\frac{4(N-1)R}{|x|}\left(\frac{1}{2}-\frac{1}{p}\right)\left(|x|^{-(N-\gamma)}\ast|v|^p\right)|v|^p\,dx.
	\end{align}
	We are left with the term in \eqref{3}, which we write as
	\begin{align}\label{morterm6a}
		\frac{2(N-\gamma)c}{p}&\int\int_{\Omega}\left[\left(1-\frac{1}{2}\frac{R}{|x|}\right)x-\left(1-\frac{1}{2}\frac{R}{|y|}\right)y\right]\frac{(x-y)|v(x)|^p|v(y)|^p}{|x-y|^{N-\gamma+2}}\,dxdy\\\label{morterm6b}
		&-\frac{4(N-\gamma)}{p}\int_{|x|>\frac{R}{2}}^{}\left(|x|^{-(N-\gamma)}\ast|v|^p\right)|v|^p\,dx
	\end{align} 
	Here,
	\begin{align*}
	\Omega=\{(x,y)\in\R^N\times\R^N\,:\,|x|> R/2\}\cup\{(x,y)\in\R^N\times\R^N\,:\,|y|> R/2\}.
	\end{align*}		
	We define $\chi_R:=\chi\left(\frac{x}{R}\right)$ for $R>0$. Then we can write the first term in \eqref{morterm1gH} as
	\begin{align}\label{morterm1a}
	\int_{|x|\leq \frac{R}{2}}^{}\chi_R^2|\nabla v|^2=\int_{|x|\leq \frac{R}{2}}^{}|\nabla(\chi_R v)|^2+\int_{|x|\leq \frac{R}{2}}^{}\chi_R\Delta(\chi_R) |v|^2
	\end{align}
	and, thus, \eqref{morterm1gH} can be written as
	\begin{align}\label{morterm1gHa}
	8&\int_{|x|\leq \frac{R}{2}}^{}|\nabla(\chi_R v)|^2+8\int_{|x|\leq \frac{R}{2}}^{}\chi_R\Delta(\chi_R) |v|^2\\\label{morterm1gHb}
	&-\frac{4(N(p-1)-\gamma)}{p}\int_{|x|\leq \frac{R}{2}}^{}\left(|x|^{-(N-\gamma)}\ast|\chi_Rv|^p\right)|\chi_Rv|^p\,dx.	
	\end{align}
	Adding \eqref{morterm4} and \eqref{morterm5}, we estimate \begin{align}\label{morterm4+5}
		\eqref{morterm4}+\eqref{morterm5}\geq-\frac{4(N-1)(p-2)}{p}\int_{|x|>\frac{R}{2}}^{}\left(|x|^{-(N-\gamma)}\ast|v|^p\right)|v|^p\,dx.
	\end{align}
	Now combining \eqref{morterm6b} with \eqref{morterm4+5} and putting together this with \eqref{morterm1gHa}, \eqref{morterm1gHb} and \eqref{morterm6a}, we obtain the following estimate
\begin{align*}
\frac{d}{dt}\mathcal{M}_{gH}(t)\geq&\,\, 8\int_{|x|\leq \frac{R}{2}}^{}|\chi_R\nabla v|^2 - \frac{4(N(p-1)-\gamma)}{p} \int_{|x|\leq \frac{R}{2}}^{}\left(|x|^{-(N-\gamma)}\ast|\chi_Rv|^p\right)|\chi_Rv|^p\\
&+c_1\int\int_{\Omega}\left[\left(1-\frac{1}{2}\frac{R}{|x|}\right)x-\left(1-\frac{1}{2}\frac{R}{|y|}\right)y\right]\frac{(x-y)|v(x)|^p|v(y)|^p}{|x-y|^{N-\gamma+2}}\,dxdy\\
&-c_2\int_{|x|>\frac{R}{2}}^{}\left(|x|^{-(N-\gamma)}\ast|v|^p\right)|v|^p-\frac{1}{R^2}M[v],
\end{align*}
where $c_1,c_2>0$ are some constants. Using Lemma \ref{coercivityballgH_new}, we obtain
	\begin{align*}
	\frac{d}{dt}\mathcal{M}_{gH}(t)\geq&\,\, 8\,\delta\, \|\chi_Rv\|^2_{L^{\frac{2Np}{N+\gamma}}}-c_2\int_{|x|>\frac{R}{2}}^{}\left(|x|^{-(N-\gamma)}\ast|v|^p\right)|v|^p-\frac{1}{R^2}M[v]\\
	&+c_1\int\int_{\Omega}\left[\left(1-\frac{1}{2}\frac{R}{|x|}\right)x-\left(1-\frac{1}{2}\frac{R}{|y|}\right)y\right]\frac{(x-y)|v(x)|^p|v(y)|^p}{|x-y|^{N-\gamma+2}}\,dxdy.
	\end{align*}
		Now we estimate the term
		\begin{align}\label{morterm7}
		\int\int_{\Omega}\left(\left(1-\frac{1}{2}\frac{R}{|x|}\right)x-\left(1-\frac{1}{2}\frac{R}{|x|}\right)y\right)\frac{(x-y)|v(x)|^p|v(y)|^p}{|x-y|^{N-\gamma+2}}\,dxdy.	
		\end{align}		
		The key to estimate the above integral is the radial Sobolev inequaity (Lemma \ref{radsob}). We divide the integral in \eqref{morterm7} into two regions
		\begin{itemize}
			\item Region I: In this region, we consider
			$$
			|x|> \frac{R}{2},\,\,|y|> \frac{R}{2},
			$$
			and observe that
			\begin{align*}
				\left|\left(1-\frac{1}{2}\frac{R}{|x|}\right)x-\left(1-\frac{1}{2}\frac{R}{|y|}\right)y\right|\lesssim |x-y|,
			\end{align*}
			then
			\begin{align}\label{morterm8}\notag
			\int_{\substack{|x|>R/2\\|y|>R/2} }\left(|x|^{-(N-\gamma)}\ast|v|^p\right)|v|^p\,dx&\lesssim \||x|^{-(N-\gamma)}\ast|v|^p\|_{L^{\frac{2N}{N-\gamma}}}\|v\|^p_{L^{\frac{2Np}{N+\gamma}}}\,\,\,\text{(H\"older's inequality)}\\\notag
			&\lesssim \|v\|^{2p}_{L^{\frac{2Np}{N+\gamma}}}\,\,\,\text{(Lemma \ref{HLS})}\\
			&\lesssim \frac{1}{R^{\frac{(N-1)(N(p-1)-\gamma)}{N}}}\|v\|_{L^2}^{\frac{N(p+1)+\gamma}{N}}\|\nabla v\|_{L^2}^{\frac{N(p-1)-\gamma}{N}}\\\notag
			&\lesssim\epsilon\quad\text{(for $R$ large enough)},
			\end{align}
			where the last inequality follows from Lemma \ref{radsob} (radial Sobolev inequality).
			\medskip\item Region II: We consider two cases:
			\begin{itemize}
				\item Case (a):	$|x|\ll|y|\approx |x-y|,\quad |y|>\frac{R}{2}$ and $|x|<\frac{R}{2}$. In this case \eqref{morterm7} becomes
				$$
				\int\int\frac{1}{|x-y|^{N-\gamma}}\,\chi_{|y|>\frac{R}{2}}|v(y)|^p\,|v(x)|^p\,dxdy,
				$$
				since
				\begin{align*}
					\left|x\left(1-\frac{R}{2}\frac{1}{|x|}\right)-y\left(1-\frac{R}{2}\frac{1}{|y|}\right)\right|&\leq \left|x\left(1-\frac{R}{2}\frac{1}{|x|}\right)\right|+\left|y\left(1-\frac{R}{2}\frac{1}{|y|}\right)\right|\\
					&\leq|x|\left(1-\frac{R}{2}\frac{1}{|x|}\right)+|y|\left(1-\frac{R}{2}\frac{1}{|y|}\right)\\
					&\lesssim|y|\approx |x-y|\quad\left(1-\frac{R}{2|y|}>1\,\,\,\text{and}\,\,\,1-\frac{R}{2|x|}<1\right).
				\end{align*}
				Again, using H\"older's inequality, radial Sobolev inequality (Lemma \ref{radsob}), and Hardy-Littlewood-Sobolev inequality (Lemma \ref{HLS}) as in the estimate for \eqref{morterm8}, we bound \eqref{morterm7} by
				\begin{align*}
				\frac{1}{R^{\frac{(N-1)(N(p-1)-\gamma)}{N}}}\|v\|_{L^2}^{\frac{N(p+1)+\gamma}{N}}\|\nabla v\|_{L^2}^{\frac{N(p-1)-\gamma}{N}}\lesssim\epsilon\quad\text{(for $R$ large)}.
				\end{align*}
				\item Case (b): $|y|\ll|x|\approx |x-y|,\quad |x|>\frac{R}{2}$ and $|y|<\frac{R}{2}$. This case is symmetric and treated with a similar argument as in Case (a).
			\end{itemize}		
		\end{itemize}
		Therefore, the contribution of \eqref{morterm7} can be made small enough for large radius.	Thus, we obtain
		\begin{align*}
			\frac{d}{dt}\mathcal{M}_{gH}(t)			\geq&\,\,8 \,\delta\, \|\chi_Rv\|^2_{L^{\frac{2Np}{N+\gamma}}}-\int_{|x|>\frac{R}{2}}^{}\left(|x|^{-(N-\gamma)}\ast|v|^p\right)|v|^p\,dx-\frac{1}{R^2}M[v].
		\end{align*}
		We rearrange the aboves inequality to write
			\begin{align}\label{morterm9_new}
			8\,\delta\,\|\chi_Rv\|^2_{L^{\frac{2Np}{N+\gamma}}}\lesssim \frac{d}{dt}\mathcal{M}_{gH}(t)+\int_{|x|>\frac{R}{2}}^{}\left(|x|^{-(N-\gamma)}\ast|v|^p\right)|v|^p+\frac{1}{R^2}M[v].
			\end{align}
		  Applying the fundamental theorem of calculus to \eqref{morterm9_new} on $[0,T]$, we obtain
		\begin{align}\label{gHtreat}\notag
		8\,\delta\int_{0}^{T}&\left(\int|\chi_Rv(x,t)|^{\frac{2Np}{N+\gamma}}\right)^{\frac{N+\gamma}{Np}}\\
		&\lesssim \sup_{t\in[0,T]}\,|\mathcal{M}_{gH}(t)|+\int_{0}^{T}\int_{|x|>\frac{R}{2}}^{}\left(|x|^{-(N-\gamma)}\ast|v|^p\right)|v|^p+\frac{T}{R^2} M[v].
		\end{align}
		Applying the radial Sobolev inequality (Lemma \ref{radsob}) along with H\"older's and Hardy-Littlewood-Sobolev inequality (Lemma \ref{HLS}), we have
		$$
		\int_{|x|>\frac{R}{2}}^{}\left(|x|^{-(N-\gamma)}\ast|v|^p\right)|v|^p\,dx\lesssim\frac{1}{R^{\frac{(N-1)(N(p-1)-\gamma)}{N}}}.
		$$
		Therefore, we obtain
			$$
			\frac{1}{T}\int_{0}^{T}\left(\int_{|x|\leq R}^{}|v(x,t)|^{\frac{2Np}{N+\gamma}}\,dx\right)^{\frac{N+\gamma}{Np}}\,dt\lesssim \frac{R}{T}+\frac{1}{R^{\frac{(N-1)(N(p-1)-\gamma)}{N}}}+\frac{1}{R^2}.
			$$
		Observe that
		$$
		\begin{cases}
		\frac{(N-1)(N(p-1)-\gamma)}{N}< 2, \quad & \text{for}\,\,s<\frac{1}{2}\\
		\frac{(N-1)(N(p-1)-\gamma)}{N}\geq 2, \quad & \text{for}\,\,s\geq\frac{1}{2}
		\end{cases}.
		$$
		Thus,
		$$
\frac{1}{T}\int_{0}^{T}\left(\int_{|x|\leq R}^{}|v(x,t)|^{\frac{2Np}{N+\gamma}}\,dx\right)^{\frac{N+\gamma}{Np}}\,dt\lesssim_{v,\delta}\begin{cases}
		\frac{R}{T}+\frac{1}{R^{\frac{(N-1)(N(p-1)-\gamma)}{N}}},\quad &\text{if}\,\,s<\frac{1}{2}\\
		\frac{R}{T}+\frac{1}{R^2},\quad &\text{if}\,\, s\geq \frac{1}{2}
		\end{cases},
		$$
		as desired.
\end{proof}
Now we prove that the energy  escapes to spatial infinity as $t\rightarrow\infty$.
\begin{proposition}[Energy evacuation, gH]\label{energyevac_gH}
	There exists a sequence of times $t_n\rightarrow\infty$ and a sequence of radii $R_n\rightarrow\infty$ such that
	$$
	\lim\limits_{n\rightarrow\infty}\left(\int_{|x|\leq R_n}^{}|v(x,t_n)|^{\frac{2Np}{N+\gamma}}\,dx\right)^{\frac{N+\gamma}{Np}}=0.	$$
\end{proposition}
\begin{proof}
	For large $T>0$ and $R>R_0$ with $R=T^{\frac{N}{N+(N-1)(N(p-1)-\gamma)}}$ for $s<\frac{1}{2}$ and $R=T^{1/3}$ for $s\geq \frac{1}{2}$ and Proposition \ref{morawetz_est_gH}, we obtain that for $s<\frac{1}{2}$
	$$
	\frac{1}{T}\int_{0}^{T}\left(\int_{|x|\leq t^{\frac{N}{N+(N-1)(N(p-1)-\gamma)}}}^{}|v(x,t)|^{\frac{2Np}{N+\gamma}}\,dx\right)^{\frac{N+\gamma}{Np}}\,dt\lesssim T^{-\frac{(N-1)(N(p-1)-\gamma)}{N+(N-1)(N(p-1)-\gamma)}}
	$$
	and for $s\geq \frac{1}{2}$
	$$
	\frac{1}{T}\int_{0}^{T}\left(\int_{|x|\leq t^{1/3}}^{}|v(x,t)|^{\frac{2Np}{N+\gamma}}\,dx\right)^{\frac{N+\gamma}{Np}}\,dt\lesssim T^{-2/3}.
	$$
	This implies that (following the similar argument as in Proposition \ref{energyevac}) there exist sequences $t_n\rightarrow\infty$ and $R_n\rightarrow\infty$ such that the energy evacuation happens on a sequence of balls as desired.
\end{proof}
\begin{proof}[Proof of Theorem \ref{main2}]
	We have (by Section \ref{variational_gH}) that $v$ is global and uniformly bounded in $H^1$. Choose $\varepsilon$ and $R$ as in Lemma \ref{Scat_crit_gH} with $t_n\rightarrow\infty$ and $R_n\rightarrow\infty$ as in Proposition \ref{energyevac_gH}. Now taking $n$ large so that $R_n\geq R$, H\"older's inequality gives
	$$
	\int_{|x|\leq R}^{}|v(x,t_n)|^2\,dx\lesssim R^{\frac{2s(p-1)+2}{p}}\left(\int_{|x|\leq R_n}^{}|v(x,t_n)|^{\frac{2Np}{N+\gamma}}\,dx\right)^\frac{N+\gamma}{Np}\rightarrow 0 \quad \text{as}\quad n\rightarrow \infty
	$$
	Therefore, Lemma \ref{Scat_crit_gH} implies that $v$ scatters in $H^1(\R^N)$ forward in time.
\end{proof}

\appendix
\section{Modifications to the proof of Lemma \ref{Scat_crit}}\label{OM}
\noindent {\it Proof of Lemma \ref{Scat_crit} for $0<s<1$.} To accommodate the full intercritical range, we substitute symmetric pair in \eqref{OM-change} with the following pair   
$$
\left(\frac{2(p+1)(p-1)}{2(p+1)-N(p-1)},p+1\right).
$$
 For reader's convenience we provide details below.

	Let $0<\varepsilon <1$ be a small constant and $R(\varepsilon)\gg 1$ be a large number, both to be chosen later. By a similar argument as in Lemma \ref{Scat_crit}, we have for (large enough) time $T_0 > \varepsilon^{-\frac{p-1}{2}}$ $> 1$ 
	\begin{align}\label{linearcomp-M}
	\|e^{it\Delta}\,u_0\|_{L^{\frac{2(p+1)(p-1)}{2(p+1)-N(p-1)}}_{t}L_x^{p+1}(\R\times\R^N)} \lesssim \varepsilon.	
	\end{align}
	We set $0<\alpha<p-1$ so that $T_1-\varepsilon^{-\alpha}>0$ and for large enough $R\gg0$, we get (by similar arguments as in Lemma \ref{Scat_crit}) 
	\begin{equation}\label{localmass-M}
	\sup_{t\in[T_1-\varepsilon^{-\alpha},T_1]}\int\chi_R(x)|u(x,t)|^2\,dx\lesssim\varepsilon.
	\end{equation}
	From \eqref{linearcomp-M} we note that the contribution from the linear component in \eqref{duhamelexp2} is small. For the second term in \eqref{duhamelexp2}, we use a similar argument as in the proof of Lemma \ref{Scat_crit} to obtain 
	\begin{align*}
	\|\int_{I_1}^{}e^{i(t-s)\Delta}(|u|^{p-1}u)(s)\,ds&\|_{L^{\frac{2(p+1)(p-1)}{2(p+1)-N(p-1)}}_{t}L_x^{p+1}([T_1,\infty)\times\R^N)}\\
	&\lesssim\||u|^{p-1}u\|_{L_{t}^{\frac{2(p+1)(p-1)}{2(p+1)p-N(p-1)p}}L_x^{\frac{p+1}{p}}(I_1\times\R^N)}\\
	&\lesssim\|u\|^p_{L_{t}^{\frac{2(p+1)(p-1)}{2(p+1)-N(p-1)}}L_x^{p+1}(I_1\times\R^N)}.
	\end{align*}
	Here, the pair $\Big(\frac{2(p+1)(p-1)}{2(p+1)p-N(p-1)p},\frac{p+1}{p}\Big)$ is a H\"older conjugate for  $\Big(\frac{2(p+1)(p-1)}{N(p-1)p-2(p+1)},p+1\Big)$, which is an $\dot{H}^{-s}$-admissible pair.
	Using H\"older's inequality in time, we get
	\begin{align*}
	\|\int_{T_1-\varepsilon^{-\alpha}}^{T_1}e^{i(t-s)\Delta}(|u|^{p-1}u)(s)\,ds&\|_{L^{\frac{2(p+1)(p-1)}{2(p+1)-N(p-1)}}_{t}L_x^{p+1}([T_1,\infty)\times\R^N)}\\
	&\lesssim|I_1|^{\frac{2(p+1)p-Np(p-1)}{2(p+1)(p-1)}}\|u\|^p_{L_{t}^{\infty}L_x^{p+1}(I_1\times\R^N)}.
	\end{align*}
	Using H\"older's inequality, %\eqref{smallbnd1},
	\eqref{localmass-M}, and Lemma \ref{radsob}, we have
	\begin{align*}
	\|u\|_{L^{\infty}_tL^{p+1}_x(I_1\times\R^N)}&=\|u\|_{L^{\infty}_tL^{p+1}_x(I_1\times B(0,R/2))}+\|u\|_{L^{\infty}_tL^{p+1}_x(I_1\times \R^N\backslash B(0,R/2))}\\
	&\lesssim\|\chi_Ru\|^{\frac{2(p+1)-N(p-1)}{2(p+1)}}_{L_t^{\infty}L_x^2(I_1\times\R^N)}\|u\|^{\frac{N(p-1)}{2(p+1)}}_{L_t^{\infty}L_x^{\frac{2N}{N-2}}(I_1\times \R^N)}	\\
	&\qquad\quad+ \|(1-\chi_R)u\|^{\frac{p-1}{p+1}}_{L_t^{\infty}L_x^{\infty}(I_1\times\R^N)}\|u\|^{\frac{2}{p+1}}_{L_t^{\infty}L_x^{2}(I_1\times\R^N)}\\
	&\lesssim\varepsilon^{\frac{2(p+1)-N(p-1)}{2(p+1)}}+R^{-\frac{p-1}{p+1}}\lesssim\varepsilon^{\frac{2(p+1)-N(p-1)}{2(p+1)}},
	\end{align*}
	where the last inequality follows by choosing $R > 0$ %=R(\varepsilon)$ 
	such that $R^{-\frac{p-1}{p+1}}\ll\varepsilon^{\frac{2(p+1)-N(p-1)}{2(p+1)}}$. 
	Therefore, we obtain
	\begin{align*}%\label{bound2-M}\notag
	\|\int_{T_1-\varepsilon^{-\alpha}}^{T_1}e^{i(t-s)\Delta}(|u|^{p-1}u)(s)\,ds&\|_{L^{\frac{2(p+1)(p-1)}{2(p+1)-N(p-1)}}_{t}L_x^{p+1}([T_1,\infty)\times\R^N)}\\\notag
	\lesssim&\,\,\varepsilon^{\frac{2(p+1)p-Np(p-1)}{2(p+1)}}|I_1|^{\frac{2(p+1)p-Np(p-1)}{2(p+1)(p-1)}}\\
	\lesssim&\,\,\varepsilon^{\frac{2(p+1)p-N(p-1)p}{2(p+1)}\left(1-\frac{\alpha}{p-1}\right)}.
	\end{align*}
	Take $0<\beta_1< \frac{2(p+1)p-N(p-1)p}{2(p+1)}\left(1-\frac{\alpha}{p-1}\right)$, then  
	\begin{align}\label{2ndterm-M}
	\|\int_{T_1-\varepsilon^{-\alpha}}^{T_1}e^{i(t-s)\Delta}(|u|^{p-1}u)(s)\,ds\|_{L^{\frac{2(p+1)(p-1)}{2(p+1)-N(p-1)}}_{t}L_x^{p+1}([T_1,\infty)\times\R^N)}\lesssim\varepsilon^{\beta_1}.
	\end{align}
	We split $F_2(u(t))$ from \eqref{duhamelexp2} via interpolation for $1<s_1<\frac{N}{2}$ as
	\begin{align}\label{F2-interpolation}\notag
	\|F_2(u)&\|_{L^{\frac{2(p+1)(p-1)}{2(p+1)-N(p-1)}}_{t}L_x^{p+1}([T_1,\infty)\times\R^N)}\\
	&\lesssim\|F_2(u)\|^{\frac{s_1-s}{s_1}}_{L_{t}^{q_1}L_x^{r_1}([T_1,\infty)\times\R^N)}\|F_2(u)\|^{\frac{s}{s_1}}_{L_{t}^{q_2}L_x^{\infty}([T_1,\infty)\times\R^N)},
	\end{align}
	where, $q_2=\frac{4}{N-2s_1}$ and $(q_1,r_1)$ is an $L^2$-admissible pair, since $N\geq 3$ and $s_1>1$ it follows that  $2<r_1<\frac{2N}{N-2}$.
	By the dispersive estimate \eqref{dispest} for $t\in[T_1,\infty)$, we bound
	\begin{align*}%\label{Ng2-M}
	\|F_2(u(t))\|_{L^{\infty}_x(\R^N)}\lesssim\int_{0}^{T_1-\varepsilon^{-\alpha}}(t-s)^{-\frac{N}{2}}\|u(s)\|^{p}_{L^{p}(\R^N)}ds.
	\end{align*}
	Observe that $p\geq 2$, we have $H^1(\R^N)\hookrightarrow L^{p}(\R^N)$ and the same is valid for $1<p<2$ from Gagliardo-Nirenberg interpolation inequality,
	and thus, by \eqref{unibnd} we obtain (for $\frac{1}{q_2}<\frac{N-2}{4}<\frac{N-2}{2}$ since $s_1>1$)
	\begin{align*}%\label{dimrem1-M}
	\|F_2(u(t))\|_{L_t^{q_2}L^{\infty}_x(\R^N)} 
	& \lesssim 
	\varepsilon^{\alpha\left(\frac{N-2}{2}-\frac{1}{q_2}\right)}=\varepsilon^{\alpha\left(\frac{N-4+2s_1}{4}\right)},
	\end{align*}
	thus, using $2s_1<N$, we  get
	\begin{align}\label{E:forbeta2-M}
	\|F_2(u)\|_{L_{t}^{q_2}L_x^{\infty}([T_1,\infty)\times\R^N)}\lesssim
	{\varepsilon^{\frac{\alpha(N-2)}{2}}}.
	\end{align}
	Again using a similar argument (as in Lemma \ref{Scat_crit}) for the other term in the interpolation \eqref{F2-interpolation}, we have
	$$
	\|F_2(u)\|_{L^{q_1}_{t}L_x^{r_1}([T_1,\infty),\R^N)}\lesssim \|u(T_1-\varepsilon^{-\alpha})\|_{L^2(\R^N)}+\|u(0)\|_{L^2(\R^N)}\lesssim 1.
	$$
	Take $0<\beta_2< \frac{\alpha(N-2)s}{2s_1}$, then from \eqref{E:forbeta2-M}, we estimate $F_2$ as
	\begin{align}\label{3rdterm-M}
	\|F_2(u)\|_{L^{\frac{2(p+1)(p-1)}{2(p+1)-N(p-1)}}_{t}L_x^{p+1}([T_1,\infty)\times\R^N)}\lesssim\varepsilon^{\beta_2}.
	\end{align}
	Putting together \eqref{linearcomp-M}, \eqref{2ndterm-M}, and \eqref{3rdterm-M} gives
	\begin{align}\label{linear-bound}
	\|e^{i(t-T_1)\Delta}u(T_1)\|_{L^{\frac{2(p+1)(p-1)}{2(p+1)-N(p-1)}}_{t}L_x^{p+1}([T_1,\infty)\times\R^N)}\lesssim\varepsilon^{\mu},
	\end{align}
	where $\mu \leq \min \{\beta_1,\beta_2\}>0$.
	Note that on $[T_1,t]$, we have
	\begin{align*}%\label{Duhamelexp2-M}
	u(t)=e^{i(t-T_1)\Delta}u(T_1)-i\int_{T_1}^{t}e^{i(t-s)\Delta}|u|^{p-1}u(s)\,ds.
	\end{align*}
	From \eqref{linear-bound} and Strichartz estimate, we observe that
	$$
	\|u \|_{L^{\frac{2(p+1)(p-1)}{2(p+1)-N(p-1)}}_{t}L_x^{p+1}([T_1,\infty)\times\R^N)}\lesssim\varepsilon^{\mu}+ \|\nabla( |u|^{p-1}u) \|_{L_t^{\frac{4(p+1)}{4(p+1-N(p-1))}}L_{x}^{\frac{p+1}{p}}([T_1,+\infty)\times\R^N)}.
	$$
	Applying the product rule and H\"older's inequality, we 
	\begin{align*}
	\|\nabla (|u|^{p-1}u)& \|_{L_t^{\frac{4(p+1)}{4(p+1-N(p-1))}}L_{x}^{\frac{p+1}{p}}([T_1,+\infty)\times\R^N)}\\
	&\lesssim \|u \|^{p-1}_{L^{\frac{2(p+1)(p-1)}{2(p+1)-N(p-1)}}_{t}L_x^{p+1}([T_1,\infty)\times\R^N)} 
	\|\nabla u \|_{L_{t}^{\frac{4(p+1)}{N(p-1)}}L_x^{p+1}([T_1,+\infty)\times\R^N)}.
	\end{align*}
	Thus, from \eqref{unibnd}, \eqref{3rdterm-M} and using a standard continuity argument on the nonlinear flow, we observe that
	\begin{equation}\label{scatbnd-M}
	\|u\|_{L^{\frac{2(p+1)(p-1)}{2(p+1)-N(p-1)}}_{t}L_x^{p+1}([T_1,\infty)\times\R^N)}\lesssim\varepsilon^{\mu}.
	\end{equation}
	Now, we consider $u_+$ (defined in \eqref{finalstate}) together with \eqref{Duhamelexp2} to get
	$$
	u(t)-e^{it\Delta}u_+=i\int_{t}^{+\infty}e^{i(t-s)\Delta}F(u(s))\,ds
	$$		
	%\end{align*}
	for all $t\geq T_1$. Therefore, estimating the $H^1$ norm
	\begin{align*}
	\|u(t)-e^{it\Delta}u_+\|_{H^1}&\lesssim \left\|\int_{t}^{+\infty}e^{i(t-s)\Delta}(1+|\nabla|)F(u(s))\,ds\right\|_{S(L^2)}\\
	&\lesssim\|(1+|\nabla|)F(u)\|_{S'(L^2;[t,+\infty))}\\
	&\lesssim \|u \|^{p-1}_{L^{\frac{2(p+1)(p-1)}{2(p+1)-N(p-1)}}_{t}L_x^{p+1}([T_1,\infty)\times\R^N)} 
	\|\nabla u \|_{L_{t}^{\frac{4(p+1)}{N(p-1)}}L_x^{p+1}([T_1,+\infty)\times\R^N)}.
	\end{align*} 
	By \eqref{unibnd} and \eqref{scatbnd-M} we observe that the Strichartz norm on $[T_1,+\infty)$ for the above expression is bounded, therefore, the tail has to vanish as $t\rightarrow+\infty$. Hence,
	$$
	\lim\limits_{t\rightarrow+\infty}||u(t)-e^{it\Delta}u_+||_{H^1(\R^N)}=0.
	$$
	\qed

\section{Modications to the proof of Lemma \ref{Scat_crit_gH}}\label{OM3}
\noindent {\it Proof of Lemma \ref{Scat_crit_gH} for $0<s<1$.} Similar to NLS, we substitute the symmetric pair $\Big(\frac{2(N+2)}{N-2s},\frac{2(N+2)}{N-2s}\Big)$ with the following pair
$$
\left(\frac{2p}{1-s},\frac{2Np}{N+\gamma}\right).
$$
Further changes are as follows. 

 As pointed out before (in Lemma \ref{Scat_crit_gH}), the proof is similar to the previous Lemma, except for the estimates of $F_1(t)$ and $F_2(t)$ (defined at the start in the proof of Lemma \ref{Scat_crit_gH}).
	To derive the estimate for $F_1(t)$, we proceed as follows (using the same argument as in Lemma \ref{Scat_crit_gH})
	\begin{align*}
	\|F_1(v)\|_{L_{t}^{\frac{2p}{1-s}}L_x^{\frac{2Np}{N+\gamma}}([T_1,\infty)\times\R^N)}&\lesssim \| (|x|^{-(N-\gamma)}\ast|v|^p)|v|^{p-2}v\|_{{L_t^{\frac{2p}{(2p-1)(1-s)}}L_x^{\frac{2Np}{2Np-N-\gamma}}}(I_1\times\R^N)}
	\end{align*}
	for $t\in I_1=[T_1-\varepsilon^{-\alpha},T_1]$, where the pair $\Big(\frac{2p}{(2p-1)(1-s)},\frac{2Np}{2Np-N-\gamma}\Big)$ is a H\"older conjugate for  $\Big(\frac{2p}{1+s(2p-1)},\frac{2Np}{N+\gamma}\Big)$, which is an $\dot{H}^{-s}$-admissible pair.	Using the H\"older's together with Hardy-Littlewood-Sobolev Lemma \ref{HLS}, we get
	\begin{align}\label{F1-M}\notag
	\|(&|x|^{-(N-\gamma)}\ast|v|^p)|v|^{p-2}v\|_{L_t^{\frac{2p}{(2p-1)(1-s)}}L_x^{\frac{2Np}{2Np-N-\gamma}}(I_1\times\R^N)}\\\notag
	\lesssim&\,\||x|^{-(N-\gamma)}\ast|v|^p\|_{L^{\frac{2}{1-s}}_tL^{\frac{2N}{N-\gamma}}_x(I_1\times\R^N)}\|v\|^{p-1}_{L^{\frac{2p}{1-s}}_tL^{\frac{2Np}{N+\gamma}}_x(I_1\times\R^N)}\\
	\lesssim&\,\|v\|^{2p-1}_{L^{\frac{2p}{1-s}}_tL^{\frac{2Np}{N+\gamma}}_x(I_1\times\R^N)}.
	\end{align}
	Using H\"older's inequality in time, we have
	$$
	\|F_1(v)\|_{L_{t}^{\frac{2p}{1-s}}L_x^{\frac{2Np}{N+\gamma}}([T_1,\infty)\times\R^N)}\lesssim |I_1|^{\frac{(2p-1)(1-s)}{2p}} \|v\|^{2p-1}_{L^{\infty}_tL^{\frac{2Np}{N+\gamma}}_x(I_1\times\R^N)}.
	$$
	Using H\"older's inequality, %\eqref{smallbnd1},
	\eqref{localmass-M}, and Lemma \ref{radsob}, we have
	\begin{align*}
	&\|u\|_{L^{\infty}_tL^{\frac{2Np}{N+\gamma}}_x(I_1\times\R^N)}=\|u\|_{L^{\infty}_tL^{\frac{2Np}{N+\gamma}}_x(I_1\times B(0,R/2))}+\|u\|_{L^{\infty}_tL^{\frac{2Np}{N+\gamma}}_x(I_1\times \R^N\backslash B(0,R/2))}\\
	&\lesssim\|\chi_Ru\|^{\frac{2p+\gamma-N(p-1)}{2p}}_{L_t^{\infty}L_x^2(I_1\times\R^N)}\|u\|^{\frac{N(p-1)-\gamma}{2p}}_{L_t^{\infty}L_x^{\frac{2N}{N-2}}(I_1\times \R^N)}	+ \|(1-\chi_R)u\|^{\frac{N(p-1)-\gamma}{Np}}_{L_t^{\infty}L_x^{\infty}(I_1\times\R^N)}\|u\|^{\frac{N+\gamma}{Np}}_{L_t^{\infty}L_x^{2}(I_1\times\R^N)}\\
	&\lesssim\varepsilon^{\frac{2p+\gamma-N(p-1)}{2p}}+R^{-\frac{N(p-1)-\gamma}{Np}}.
	\end{align*}
	Choosing $R > 0$ %=R(\varepsilon)$ 
	such that $R^{-\frac{N(p-1)-\gamma}{Np}}\ll\varepsilon^{\frac{2p+\gamma-N(p-1)}{2p}}$, we get
	\begin{align*}%\label{bound1gH}
	\|u \|_{L^{\infty}_tL^{\frac{2Np}{N+\gamma}}_x(I_1\times\R^N)}\lesssim\varepsilon^{\frac{2p+\gamma-N(p-1)}{2p}}.
	\end{align*}
	Therefore, we obtain
	$$
	\|F_1(v)\|_{L_{t}^{\frac{2p}{1-s}}L_x^{\frac{2Np}{N+\gamma}}([T_1,\infty)\times\R^N)}\lesssim\varepsilon^{\frac{2p+\gamma-N(p-1)}{2p}-\alpha\,\frac{(2p-1)(1-s)}{2p}}.
	$$
	So, for $0<\alpha<\frac{2(p-1)}{2p-1}$, we have that 
	\begin{align}\label{2ndtermgH-M}
	\|F_1(v)\|_{L_{t}^{\frac{2p}{1-s}}L_x^{\frac{2Np}{N+\gamma}}([T_1,\infty)\times\R^N)}\lesssim\varepsilon^{\beta_3},
	\end{align}
	where $\beta_3>0$ chosen in the same way as in the previous lemma. To derive the estimate for $F_2(t)$ we use interpolation to get
	\begin{align*}
	\|F_2(v)&\|_{L_{t}^{\frac{2p}{1-s}}L_x^{\frac{2Np}{N+\gamma}}([T_1,\infty)\times\R^N)}\\
	&\lesssim\|F_2(v)\|^{\frac{N+\gamma}{N+\gamma+2sp}}_{L_{t}^{\frac{2p}{1-s}}L_x^{\frac{2Np}{N+\gamma+2sp}}([T_1,\infty)\times\R^N)}\|F_2(v)\|^{\frac{2sp}{N+\gamma+2sp}}_{L_{t}^{\frac{2p}{1-s}}L_x^{\infty}([T_1,\infty)\times\R^N)}.
	\end{align*}
	Note that $\frac{2p}{1-s}>\frac{2}{N-2}$ and $\Big(\frac{2p}{1-s},\frac{2Np}{N+\gamma+2sp}\Big)$ is an $L^2$-admissible pair.
	We bound the last term above for $t\in[T_1,\infty)$ using the dispersive estimate \eqref{dispest} followed by H\"older's inequality and Lemma \ref{HLS} yields
	$$
	\|F_2(v(t))\|_{L_x^{\infty}(\R^N)}\lesssim\int_{0}^{T_1-\varepsilon^{-\alpha}}(t-s)^{-\frac{N}{2}}\|v\|_{L^{p_2}(\R^N)}^p\|v\|^{p-1}_{L^{p_1}(\R^N)}ds,
	$$
	where $p_1=\frac{2N(p-1)}{N+\gamma}$ and $p_2=\frac{2Np}{N+\gamma}$. Observe that for $N\geq 3$ and $2\leq p<\frac{N+\gamma}{N-2}$ one has the embedding $H^1(\R^N)\hookrightarrow L^{p_{i}}(\R^N)$ with $i=1,2$,  
	and thus, by \eqref{unibnd}
	\begin{align*}%\label{dimrem2-M}
	\|F_2(v(t))\|_{L_t^{\frac{2p}{1-s}}L_x^{\infty}([T_1,\infty)\times\R^N)}\lesssim \varepsilon^{\alpha\left(\frac{N-2}{2}-\frac{1-s}{2p}\right)}.
	\end{align*}
	Next, note that 
	$$
	F_2(t)=e^{i(t-T_1+\varepsilon^{-\alpha})\Delta} v(T_1-\varepsilon^{-\alpha}) - e^{it\Delta} v(0),
	$$
	and using the Strichartz estimate along with \eqref{unibnd}, we have
	$$
	\|F_2(v)\|_{L^\frac{2p}{1-s}_{t}L_x^{\frac{2Np}{N+\gamma+2sp}}([T_1,\infty)\times\R^N)}\lesssim \|v(T_1-\varepsilon^{-\alpha})\|_{L^2(\R^N)}+\|v(0)\|_{L^2(\R^N)}\lesssim 1.
	$$
	Similar to the step \eqref{3rdterm-M}, we take $\beta_4>0$ so that
	\begin{align}\label{3rdtermgH-M}
	\|F_2(v)\|_{L^\frac{2p}{1-s}_{t}L_x^{\frac{2Np}{N+\gamma}}([T_1,\infty)\times\R^N)}\lesssim\varepsilon^{\beta_4}.
	\end{align}
	Putting together \eqref{2ndtermgH-M}, and \eqref{3rdtermgH-M} along with the small contribution of linear component (similar to NLS) gives
	\begin{align}\label{boundgH-M}
	\|e^{i(t-T_1)\Delta}v(T_1)\|_{L^\frac{2p}{1-s}_{t}L_x^{\frac{2Np}{N+\gamma}}([T_1,\infty)\times\R^N)}\lesssim\varepsilon^{\nu}
	\end{align}
	for $0<\nu \leq \min \{\beta_3, \beta_4\}$. For the bound on the nonlinear solution we again consider Duhamel's formula
	\begin{align*}
	v(t)=e^{i(t-T_1)\Delta}v(T_1)-i\int_{T_1}^{t}e^{i(t-s)\Delta}F(v(s))\,ds,
	\end{align*}
	where $F(v) = (|x|^{-(N-\gamma)}\ast|v|^p)|v|^{p-2}v$. Taking $L^\frac{2p}{1-s}_{t}L_x^{\frac{2Np}{N+\gamma}}([T_1,\infty)\times\R^N)$ norms, we observe from the linear evolution bound and inhomogeneous Strichartz estimate that
	$$
	\|v\|_{L^\frac{2p}{1-s}_{t}L_x^{\frac{2Np}{N+\gamma}}([T_1,\infty)\times\R^N)}\lesssim\varepsilon^{\nu}+\|\nabla F(v)\|_{L_{t}^{\frac{2p}{2p-1-s(p-1)}}L_x^{\frac{2Np}{2Np-N-\gamma}}([T_1,\infty)\times\R^N)}.
	$$
	By H\"older's inequality, product rule and Lemma \ref{HLS}, we obtain
	\begin{align*}
	\|\nabla\big((|x|^{-(N-\gamma)}&\ast|v|^p)|v|^{p-2}v\big)\|_{L_{t}^{\frac{2p}{2p-1-s(p-1)}}L_x^{\frac{2Np}{2Np-N-\gamma}}([T_1,\infty)\times\R^N)}\\
	&
	\lesssim\,\|v\|^{2(p-1)}_{L^{\frac{2p}{1-s}}_{t}L_x^{\frac{2Np}{N+\gamma}}([T_1,\infty)\times\R^N)}\|\nabla v\|_{L^{\frac{2p}{1+s(p-1)}}_{t}L_x^{\frac{2Np}{N+\gamma}}([T_1,\infty)\times\R^N)}.
	\end{align*}
	%REMARK: Check the above estimate
	The remaining proof follows similar reasoning as in Lemma \ref{Scat_crit}, completing the proof.\qed

%%%%%%%%%%%%%%%%%%%%%%%

%/Dropbox/Research-Andy/gHartree_New proof Scattering/

%\clearpage
%\newpage

\bibliography{Andy-references}
\bibliographystyle{abbrv}

\end{document}